\documentclass[11pt,a4paper]{article}

\usepackage{authblk}

\makeatletter
\newcommand{\subjclass}[2][2010]{%
  \let\@oldtitle\@title%
  \gdef\@title{\@oldtitle\footnotetext{#1 \emph{Mathematics subject classification.} #2}}%
}
\newcommand{\keywords}[1]{%
  \let\@@oldtitle\@title%
  \gdef\@title{\@@oldtitle\footnotetext{\emph{Key words and phrases.} #1.}}%
}
\makeatother

\AtEndDocument{\bigskip{\footnotesize%
  \textit{E-mail address}, Xiaoqin Guo: \texttt{guoxq@ucmail.uc.edu} 
}}

\usepackage{etoolbox}
\usepackage{comment,amsmath,amssymb,amsthm,bbm,mathtools,mathrsfs}
\usepackage{tikz,pgfplots,color,setspace,lmodern}
\usepackage[]{stix}
\usetikzlibrary{shapes, backgrounds,patterns,decorations.pathreplacing}
\usepackage{multicol,float}
\usepackage{makeidx}
\usepackage{hyperref, enumerate}
\pgfplotsset{compat=newest}
\makeindex

\def\var{{\rm Var}}

\def\XXint#1#2#3{{\setbox0=\hbox{$#1{#2#3}{\int}$ }
\vcenter{\hbox{$#2#3$ }}\kern-.6\wd0}}

\newcommand{\dd}{\mathrm{d}}
\newcommand{\by}[2]{\stackrel{#1}{#2}}

\newcommand{\R} {\mathbb{R}}

\newcommand{\Z} {\mathbb{Z}}
\newcommand{\nn}{\nonumber}
\newcommand{\N} {\mathbb{N}}

\newcommand{\rn}[1]{\mathrm{#1}}  
\newcommand{\ff}{\mathcal{F}} 
\newcommand{\ww}{\mathcal{W}}
\newcommand{\zz}{\mathcal{Z}}
\newcommand{\sn}{\mathcal{N}} 
\newcommand{\orf}{\psi} 
\newcommand{\tbd}{\beta} 
\newcommand{\fcon}{\theta_\orf} 
\newcommand{\scon}{\rho}

\providecommand{\onorm}[1]{\norm{#1}_\orf} 
\providecommand{\ee}[1]{E[#1]} 
\providecommand{\bee}[1]{E\left[#1\right]}
\providecommand{\idkt}[1]{\mathbbm{1}_{#1}} 

\DeclarePairedDelimiter{\abs}{\lvert}{\rvert}

\DeclarePairedDelimiter{\norm}{\lVert}{\rVert}
\DeclarePairedDelimiter{\bnorm}{\bigg\lVert}{\bigg\rVert}
\providecommand{\Abs}[1]{\Bigr\lvert#1\Bigl\rvert}

\providecommand{\mc}[1]{\mathcal{#1}}
\providecommand{\ms}[1]{\mathscr{#1}}

\newtheorem{theorem}{Theorem}
\newtheorem{lemma}[theorem]{Lemma}
\newtheorem{corollary}[theorem]{Corollary}
\newtheorem{proposition}[theorem]{Proposition}

\theoremstyle{definition}  
\newtheorem{definition}[theorem]{Definition}
\newtheorem{remark}[theorem]{Remark}

\newtheorem{example}[theorem]{Example}
\begin{document}

\title{On the rate of convergence of the martingale central limit theorem in Wasserstein distances}

\author[1]{Xiaoqin Guo \thanks{The work of XG is supported by Simons Foundation through Collaboration Grant for Mathematicians \#852943.}}

\affil[1]
{
Department of Mathematical Sciences, 
University of Cincinnati , 2815 Commons Way, Cincinnati, OH 45221, USA}



\maketitle

\begin{abstract}
For martingales with a wide range of integrability, 
we will quantify
 the rate of convergence of the central limit theorem via Wasserstein distances  of order $r$, $1\le r\le 3$. 
Our bounds are in terms of Lyapunov's coefficients and the $\ms L^{r/2}$ fluctuation of the total conditional variances. We will show that  our  Wasserstein-1 bound is optimal up to a multiplicative constant.
\end{abstract}

\tableofcontents

\section{Introduction}
In this paper, we consider the rate of convergence of the martingale central limit theorem (CLT) under Wasserstein distances. Let $Y_\cdot=\{Y_m\}_{m=1}^n$ be a square integrable martingale difference sequence (mds)  with respect to  $\sigma$-fields $\{\ff_m\}_{m=0}^n$. Here $\ff_0$ denotes the trivial $\sigma$-field. For $1\le m\le n$, let $s_m^2=\sum_{i=1}^mE[Y_i^2]$ and
\begin{align}\label{def:notations}
\sigma_m^2(Y)&=E[Y_m^2|\ff_{m-1}],\nn\\
V_n^2&=\frac{1}{s_n^2}\sum_{i=1}^n\sigma_i^2(Y),\nn\\
X_m&=\frac{1}{s_n}\sum_{i=1}^m Y_i.
\end{align}
For a mds $\{Y_i\}_{i=1}^\infty$, when $\lim_{n\to\infty}V_n^2=1$ in probability and  Lindeberg's condition is satisfied, it is well-known that
\[
X_n\Rightarrow \sn\quad \text{ as }n\to\infty,
\]
where $\sn\sim\sn(0,1)$ denotes a standard normal random variable. 

To quantify such a convergence in distribution, 
one of the most important metrics is the Wasserstein distance which has deep roots in the optimal transport theory \cite{Villani}.
Recall that, for two probability measures $\mu,\nu$ on $\R$, their {\it Wasserstein distance} (also called {\it minimal distance}) $\ww_r$ of order $r$, $r\ge 1$, 
is defined by
\[
\ww_r(\mu,\nu)=\inf\{\norm{U-V}_r: U\sim\mu, V\sim\nu\},
\]
where $\norm{U}_r=E[|U|^r]^{1/r}$. 
With abuse of notations, we may use random variables as synonyms of their distributions. E.g., for $U\sim\mu, V\sim\nu$, we may write $\ww_r(\mu,\nu)$ as $\ww_r(U,V)$.

When $r=1$, recall that $\ww_1$ admits the following alternative representations:
\[
\ww_1(U,V)=\sup_{h\in\rm{Lip}_1}\left\{\ee{h(U)}-\ee{h(V)}\right\}=\int_\R\Abs{F_U(x)-F_V(x)}\dd x
\]
where $\rm{Lip}_1$ denotes the set of all 1-Lipschitz functions, and $F_U(x):=P(U\le x)$.
\medskip

Throughout the paper, we use $c, C$ to denote positive constants which may change from line to line. 
Unless otherwise stated, $C_p, c_p$ denote constants depending only on the parameter $p$.
We write $A
\lesssim B$ if $A\le CB$, and 
$A\asymp B$ if $A\lesssim B$ and $A\gtrsim B$. 
We also use notations $A\lesssim_p B$, $A\asymp_p B$ to indicate that the multiplicative constant depends on the parameter $p$.

\subsection{Earlier results in the literature}

There is an immense literature on the convergence rate of the CLT for independent random variables. 
Such results are often phrased in terms of Lyapunov's coefficient (i.e., the term in the $\ms L^p$ Lyapunov condition). To be specific, for $p\ge 1$, set\footnote{Note that the meaning of our $L_p$ notation is slightly different from that of \cite{Bikjalis-66,Hall-Heyde-80,Haeusler-88,Rio-09,Bobkov-18} in the literature whose $L_p$ means the $L_p^p$ in our paper. }
\begin{equation}\label{eq:Lyap-lp}
L_p=L_p(Y):=\frac{1}{s_n}\left(
\sum_{i=1}^n \ee{|Y_i|^p}
\right)^{1/p}.
\end{equation}
Note that typically $L_p$ is of size $O(n^{1/p-1/2})$.

When $\{Y_i\}_{i=1}^n$ is a centered independent sequence, a nonuniform estimate of 
Bikjalis \cite{Bikjalis-66} implies
$\ww_1(X_n,\sn)\le cL_p^p$, $p\in(2,3]$, extending the $\ww_1$ bounds of Esseen \cite{Esseen-58},  Zolotarev \cite{Zolo-64}, and Ibragimov \cite{Ibra-66}. For $r>2$, 
Sakhanenko \cite{Sakhanenko-85} proved that $\ww_r(X_n,\sn)\le cL_r$ which is optimal for independent $\ms L^r$-integrable variables. For $1\le r\le 2$, it is established by
Rio \cite{Rio-98, Rio-09}  that 
\begin{equation}\label{eq:bobkov}
\ww_r(X_n,\sn)
\le 
c_r L_{r+2}^{(r+2)/r}.
\end{equation}
In particular, when $\{Y_i\}_{i=1}^n$ have roughly the same $\ms L^{r+2}$ moments, the  $\ww_r$ bound in \eqref{eq:bobkov}  achieves the optimal rate $O(n^{-1/2})$. 
Bobkov \cite{Bobkov-18} confirmed Rio's conjecture that \eqref{eq:bobkov} holds for all $r\ge 1$.
 Cf. \cite{Rio-09, Bobkov-18} and references therein. 
 For further developments, see e.g. \cite{Zhai-18, FSX-19, CFP-19, Fathi-19, EMA-20, Bonis-20, LA-23} for work in the multivariate setting and e.g. \cite{Fang-19, LA-23} for results on random variables with local-dependence.
\medskip

It is natural to 
 to ask whether the martingale CLT can be quantified by similar Wasserstein metric bounds in terms of $L_p$.
 However, 
despite its  theoretical importance,  
 there are very few such results for (general) martingales compared to the independent case, let alone questions on optimal rates.

When $\{Y_i\}_{i=1}^n$ is a mds, the non-uniform bound of the distribution functions  by Joos \cite{Joos-93} (which generalizes Haeusler and Joos \cite{HJ-88}) implies that
\begin{equation}\label{eq:haeusler-joos}
\ww_1(X_n,\sn)\le c_{p,q}\left( L_p^{p/(1+p)}+\norm{V_n^2-1}_{q}^{q/(2q+1)}	\right) 
\quad \text{ for }2<p<\infty, q\ge 1.
\end{equation}
In the case $M=\max_{i=1}^n\norm{Y_i}_\infty<\infty$,  the nonuniform bound of Joos \cite{Joos-91} implies
\begin{equation}\label{eq:joos-91}
\ww_1(X_n,\sn)\le c_q\left(
\frac{M}{s_n}\log\big(e+\frac{s_n^2}{M^2}\big)+
\norm{V_n^2-1}_{q}^{q/(2q+1)}
\right) \, \text{ for }q\ge 1.
\end{equation}
Still for the $\ms L^\infty$ case, Dung, Son, Tien \cite{DST-14} improved the last term in \eqref{eq:joos-91} to be $C_q\norm{V^2-1}_{q}^{1/2}$ for any $q>1/2$. 
 Under the condition
\begin{equation}\label{eq:assume-v=1}
P(V_n^2=1)=1,
\end{equation} 
R\"ollin's \cite[(2.1)]{Rollin-18} result inferred that (See also \cite[Lemma 2.1]{FM-20})
\begin{equation}\label{eq:rollin}
\ww_1(X_n,\sn)\lesssim L_3.
\end{equation}
Fan and Ma \cite{FM-20} dropped the condition \eqref{eq:assume-v=1} and obtained
\begin{equation}\label{eq:fanma}
\ww_1(X_n,\sn)\lesssim_q
L_3+\norm{V_n^2-1}_{q}^{1/2}+\frac{1}{s_n}\max_{i=1}^n\norm{Y_i}_{2q}, \quad \forall q\ge 1.
\end{equation}
Recently,  assuming \eqref{eq:assume-v=1},  Fan, Su \cite[Corollary~2.5]{FS-22} implicitly extended \eqref{eq:rollin} to be
\begin{equation}\label{eq:fan-su}
\ww_1(X_n,\sn)\lesssim L_p \quad \text{ for }2<p\le 3.
\end{equation}

Dedecker, Merlev\`ede, and Rio \cite{DMR-22} proved  $\ww_r$ bounds, $r\in(0,3]$, that involve the quantities $\sum_{i=\ell}^nE[Y_i^2|\ff_{\ell-1}]-E[Y_i^2]$  (instead of $V_n^2-1$ in our bounds), $2\le\ell\le n$, which can be suitably bounded  in many situations.

Readers may refer to \cite{DMR-09, Dedecker-Merlevede-11, Rollin-18, FS-22, DMR-22} for Wasserstein bounds for the martingale CLT under other special conditions (e.g. the sequence being stationary, $\sigma_i$'s being close to be deterministic, $\sigma_i$'s are uniformly bounded from below, or certain variants of the $\ms L^{\infty}$ case, etc.).

\begin{remark}
Unlike the Wasserstein metric, 
 the Kolmogorov distance bounds for the martingale CLT has been thoroughly investigated since the 1970s. One of the earliest results is due to Heyde, Brown \cite{Heyde-Brown-70} which states that, for $2<p\le 4$,
\begin{equation}\label{eq:heyde-brown}
d_K(X_n,\sn):=\sup_{x\in\R}\abs{F_{X_n}(x)-F_{\sn}(x)}
\lesssim_p
L_p^{p/(p+1)}+\norm{V_n^2-1}_{p/2}^{p/(2p+2)}.
\end{equation}
When the mds is $\ms L^\infty$, i.e., $\max_{i=1}^n\norm{Y_i}_\infty=M<\infty$,  Bolthausen \cite{Bol-82} showed that 
\begin{equation}
\label{eq:bolthausen}
d_K(X_n,\sn)
\lesssim_M
\frac{n\log n}{s_n^3}+
\min\bigg\{\norm{V_n^2-1}_{\infty}^{1/2},
\norm{V_n^2-1}_{1}^{1/3}\bigg\}
\end{equation}
and that the first term $\tfrac{\log n}{n^{1/2}}$ is optimal for the case $s_n^2=n$. 
Haeusler \cite{Haeusler-84,Haeusler-88} generalized \eqref{eq:heyde-brown} to $2<p<\infty$ and showed that the first term $L_p^{p/(p+1)}$ is exact.
Joos \cite{Joos-93} proved that the second term in \eqref{eq:heyde-brown} can be replaced by $\norm{V_n^2-1}_{q}^{q/(2q+1)}$,  $q\ge 1$.  Mourrat \cite{Mourrat-13} improved the second term in \eqref{eq:bolthausen} to $(\norm{V_n^2-1}_{q}+s_n^{-2})^{q/(2q+1)}$, $q\ge 1$. 
Cf. also \cite{HJ-88, Joos-93, Gra-Haeu-00, Ouchti-05, ElM-Ouchti-07, Fan-19, FS-22} and references therein for work in this direction.
\end{remark}

\subsection{Motivation and our contributions}

Our paper  concerns  the $\ww_r$ convergence rates for the martingale CLT, $1\le r\le 3$. 

Let us comment on some weaknesses of the aforementioned $\ww_1$ bounds. 

When the martingale differences are at least $\ms L^p$-integrable, $2<p\le 3$, the best $\ww_1$ rates given by \eqref{eq:rollin}, \eqref{eq:fan-su}, \eqref{eq:fanma} are typically $n^{1/p-1/2}\ge n^{-1/6}$, leaving a big gap between the rate $O(n^{-1/2}\log n)$ in the $\ms L^\infty$ case \eqref{eq:joos-91}, not to mention that the condition \eqref{eq:assume-v=1} imposed in \eqref{eq:fan-su} is too restrictive for general martingales. 
Compared to \eqref{eq:fan-su}, the exponent of $L_p$ in \eqref{eq:haeusler-joos} is clearly not optimal, at least for  $2<p\le 3$.  
But 
\eqref{eq:haeusler-joos} does not assume \eqref{eq:assume-v=1}, and it offers a typical $\ww_1$ rate $O(n^{(2-p)/(2p+2)})$ which is better than $n^{-1/6}$ for $p>7/2$. However, the constant $c_p$ in \eqref{eq:haeusler-joos} is expected to grow linearly as $p\to\infty$, 
rendering \eqref{eq:haeusler-joos} a useless bound when  $p$ is bigger than $n^{1/2}$.
\medskip

Notice that all of the results discussed above are $\ww_1$ bounds. There are rarely any results on the $\ww_r$ bounds, $r>1$, in terms of $L_p$ for (general) martingales. 
\medskip

Can we obtain $\ww_r$ bounds, $r\ge 1$, in terms of the Lyapunov coefficient \eqref{eq:Lyap-lp}  for the martingale CLT that generalize all of the previous results \eqref{eq:haeusler-joos}, \eqref{eq:joos-91}, \eqref{eq:rollin}, \eqref{eq:fanma}, \eqref{eq:fan-su}? Further, what are the optimal rates and how to justify their optimality? 

Is it possible to get Wasserstein distance bounds for the CLT so that the constant coefficents do not blow up as the integrability of the martingale increases? 
Such estimates would be important when we have a limited sample size, and they allow us to exploit the integrability of the martingale to obtain better rates.

In theory and in applications, 
there are numerous stochastic processes that do not fit into any of the $\ms L^p$ category, $2<p<\infty$, cf. \cite{vdG-L-13,Kulik-20,VGNA-20}. 
Can we quantify the CLT for martingales with much wider spectrum of integrability than $\ms L^p$, $p\ll n^{1/2}$?


Motivated by these questions, we will prove the following results.
\begin{enumerate}[(1)]
\item We will obtain $\ww_r$, $1\le r\le 3$, convergence rates for martingales which are {\it Orlicz}-integrable. For instance, if 
\begin{equation}\label{eq:orlicz-informal}
A:=\frac{1}{n}\sum_{i=1}^n\orf(|Y_i|)<\infty
\end{equation}
for appropriate convex function $\orf$ that grows at most polynomially fast, then
\begin{equation}\label{eq:ww1-informal}
\ww_1(X_n,\sn)\lesssim_\orf\frac{A\vee1}{s_n}\orf^{-1}(n)+\norm{V_n^2-1}_{1/2}^{1/2},
\end{equation}
where $\orf^{-1}$ denotes the inverse function of $\orf$.

Our result greatly generalizes the known $\ww_1$ bounds \eqref{eq:haeusler-joos},  \eqref{eq:rollin}, \eqref{eq:fanma}, \eqref{eq:fan-su} for $\ms L^p$ martingales, $2<p\le 3$. Moreover, we will prove that both terms $\tfrac{1}{s_n}\orf^{-1}(n)$ and $\norm{V_n^2-1}_{1/2}^{1/2}$ in this $\ww_1$ bound are optimal.

\item We will derive Wasserstein bounds whose constant coefficients do not depend on the integrability of the variables.  Take the $\ww_1$ distance for example, if \eqref{eq:orlicz-informal} holds for $\orf$ in a wide class of convex functions,  we show that
\begin{equation}\label{eq:ww1-log-informal}
\ww_1(X_n,\sn)
\lesssim 
\frac{A\vee1}{s_n}\orf^{-1}(n)\log n+\norm{V_n^2-1}_{1/2}^{1/2}.
\end{equation}
This explains the presence of the $\log n$ term in \eqref{eq:joos-91}, and it implies that if the mds is  $\ms L^\infty$ bounded, i.e., $M=\max_{i=1}^n\norm{Y_i}_\infty<\infty$, then
\[
\ww_1(X_n,\sn)
\lesssim 
\frac{M}{s_n}\log(e+\frac{s_n^2}{M^2})+\norm{V_n^2-1}_{1/2}^{1/2}.
\]
The novelty of this result lies not only on the fact that it encompasses an even larger spectrum (all the way up to $\ms L^\infty$) of integrability than our first bound \eqref{eq:ww1-informal}, 
but also that it yields a better bound than \eqref{eq:ww1-informal} when the order of the ``integrability" $\orf$ is bigger than the logarithm of the sample size (i.e., beyond $O(\log n)$).

\item  Similar bounds for the $\ww_r$ distances in terms of the Lyapunov coefficient and $\norm{V_n^2-1}_{r/2}$, $1<r\le 3$, will be established as well.

\end{enumerate}

\subsection{Structure of the paper}

The organization of our paper is as follows. 

 Subsection~\ref{subsec:prel} contains definitions of N-functions and the corresponding Orlicz norm for random sequences. 
In Section~\ref{sec:main}, we present our main  Wasserstein distance bounds (Theorems~\ref{thm:ww1-new},\ref{thm:ww2}, and \ref{thm:ww3}) and the optimality of the $\ww_1$ rate (Proposition~\ref{prop:optimal}). In Section~\ref{sec:special}, using Taylor expansion and Lindeberg's telescopic sum argument, we will derive $\ww_r$ bounds in terms of the conditional moments  for martingales with $V_n^2=1$ a.s.. 
As consequences,  $\ww_1,\ww_2$ bounds (Corollary~\ref{cor:v=1}) will be obtained for martingales that satisfy certain special conditions.

Section~\ref{sec:pf-main} is devoted to the proof of our main results. 
Our proof consists of the following components.
First, we truncate the martingale as Haeusler \cite{Haeusler-88} and elongate it as Dvoretzky \cite{Dvo-72}  to turn it into a sequence with bounded increments and $V_n^2=1$. We will bound the error of this modification in terms of the Lyapunov conefficient and $(V_n^2-1)$. (See Section~\ref{subsec:modify}.)
Then, as a crucial technical step, we use Young's inequality within conditional expectations to ``decouple" the Lyapunov coefficient and $\sigma_i^2(Y)$'s from the  conditional moments, and turn the $\ww_r$ bound to be an optimization problem over three parameters. This argument is robust enough for us to deal with martingales with very flexible integrability.
 Our another key observation is that
 there should be  different bounds for
the two different scenarios when the martingale  is 
``at most $\ms L^p$" and ``more integrable than $\ms L^p$".
By possibly sacrificing a small factor  (e.g. $\log n$), we can make the constants of the bound  independent of the integrability $\orf$, which leads to  bounds of type \eqref{eq:ww1-log-informal}.

Finally, in Section~\ref{sec:optimal} we construct examples to show that our $\ww_1$ bound is optimal. We leave some open questions in Section~\ref{sec:qs}.

\subsection{Preliminaries: N-functions and  an Orlicz-norm for sequences}\label{subsec:prel}

To generalize the notion of $\mc L^p$ integrability,  we recall the definitions of N-function and the corresponding Orlicz-norm in this subsection.
\begin{definition}\label{def:N-func}
A convex function $\orf:[0,\infty)\to[0,\infty)$ 
is called an {\it N-function} if  it  satisfies $\lim_{x\to 0}\orf(x)/x=0$, $\lim_{x\to\infty}\orf(x)/x=\infty$, and $\orf(x)>0$ for $x>0$.
For two functions $\orf_1,\orf_2\in[0,\infty)^{[0,\infty)}$, we write 
\[
\orf_1\preccurlyeq\orf_2 \quad\text{ or }\quad \orf_2 \succcurlyeq\orf_1
\] if  
$\frac{\orf_2(x)}{\orf_1(x)}$ is non-decreasing on $(0,\infty)$.
\end{definition}

Note that every N-function $\orf$ satisfies $\orf\succcurlyeq x$. 
 See \cite{Adams-03}. 
 
 Denote the Fenchel-Legendre transform of $\orf$ by $\orf_*\in[0,\infty)^{[0,\infty)}$, i.e.,
\[
\orf_*(x)=\sup\{xy-\orf(y):y\in\R\}  \quad \text{ for }x\ge 0.
\]
Then $\orf_*$ is still an N-function.Young's inequality states that
\begin{equation}
\label{eq:young}
xy\le \orf(x)+\orf_*(y) \quad\text{ for all }x,y\ge 0.
\end{equation}
Another useful relation between the pair $\orf, \orf_*$ is
\begin{equation}
\label{eq:orlicz-ineq}
x\le \orf^{-1}(x)\orf_*^{-1}(x)\le 2x, \quad\forall x\in[0,\infty).
\end{equation}
See \cite{Adams-03} for a proof and for more properties of  N-functions.

\begin{definition}
\label{def:orlicz-norm}
For any N-function $\orf$ and $n\in\N$, the {\it $\ms L^\orf$-Orlicz norm} of a random sequence $Y=\{Y_m\}_{m=1}^n$ with length $n$ is defined by 
\begin{equation}\label{eq:def-orlicz-norm}
\norm{Y}_\orf:=\inf\big\{c>0: \frac1n\sum_{i=1}^n\bee{\orf(|Y_i|/c)}\le 1\big\}.
\end{equation}
In particular, when $n=1$, i.e., $Y=Y_1$ is a single random variable, we still write
 \[
 \norm{Y_1}_\orf=\inf\{c>0:\ee{\orf(|Y_1|/c)}\le 1\}.
 \] 
When $\orf(x)=x^p$, $p\ge 1$, we simply write $\onorm{Y}$ as $\norm{Y}_p$.
Notice that
\begin{align}
\norm{Y}_p&=\left(\frac{1}{n}\sum_{i=1}^n\ee{|Y_i|^p}\right)^{1/p},\label{eq:y-lp}\\
\text{and }\quad \lim_{p\to\infty}\norm{Y}_p&=\max_{i=1}^n\norm{Y_i}_\infty=:\norm{Y}_\infty.\label{eq:def-y-infty}
\end{align}
\end{definition}

\begin{remark}
\begin{enumerate}[(a)]
 \item 
Using the property $\orf(\lambda x)\ge\lambda\orf(x)$ for $\lambda\ge1$ for N-functions, it is easily seen that the integrability condition 
 \eqref{eq:orlicz-informal} implies $\onorm{Y}\le A\vee 1$.
 
\item Since this paper concerns square-integrable martingales, we are only interested in sup-quadratic N-functions $\orf\succcurlyeq x^2$. 
For instance, $x^2\log(x+1)$, $x^p$ (with $p\ge 2$), $e^x-x-1$, $\exp(x^\beta)-1$, $\exp\left(
\ln(x+1)^\beta
\right)-1$ (with $\beta>1$) are among such N-functions. 
 
\item The meaning of our notation  $\norm{Y}_p$ is different from some work in the literature, e.g. \cite{Bol-82,FS-22} where it means $\max_{i=1}^n\norm{Y_i}_p$.  Note that $\norm{Y}_\orf\le \max_{i=1}^n\norm{Y_i}_\orf$.
 \item 
Removing the $\tfrac{1}{n}$ in \eqref{eq:y-lp} only changes a multiplicative factor of $\norm{Y}_p$.   However, for general $\onorm{\cdot}$, the presence of $\tfrac{1}{n}$ in \eqref{eq:def-orlicz-norm} is crucial for the definition--removing it would drastically change the meaning of the norm.
\end{enumerate}
\end{remark}

\section{Main results}\label{sec:main}
Recall $\norm{Y}_\infty$ in \eqref{eq:def-y-infty}. For the mds $\{Y_i\}_{i=1}^n$ and an N-function $\orf$, we generalize the notation $L_p$ in \eqref{eq:Lyap-lp} to be
\begin{equation}\label{def:L-orf}
L_\orf=\frac{1}{s_n}\onorm{Y}\orf^{-1}(n),
\quad
L_\infty=\frac{1}{s_n}\norm{Y}_\infty,
\end{equation}
where $\orf^{-1}$ denotes the inverse function of $\orf$.

Our first main results are two $\ww_1$ bounds that generalize  \eqref{eq:haeusler-joos}, \eqref{eq:joos-91}, \eqref{eq:rollin}, \eqref{eq:fanma}, \eqref{eq:fan-su}.

\begin{theorem}[$\ww_1$ bounds]
\label{thm:ww1-new}
Let $\{Y_m\}_{m=1}^n$ be a martingale difference sequence. Recall the notations $X_n, V_n^2, L_\orf, L_\infty, \preccurlyeq, \succcurlyeq$ in \eqref{def:notations}, \eqref{def:L-orf}, and Definition~\ref{def:N-func}.

\begin{enumerate}[(i)]
\item\label{item:thm-ww1-new-1} For any N-function $\orf\succcurlyeq x^2$,
\begin{equation}\label{eq:w1-general}
\ww_1(X_n,\sn)
\lesssim L_\orf\log\left(e+L_\orf^{-2}\right)+\norm{V_n^2-1}_{1/2}^{1/2}.
\end{equation}

\item\label{item:thm-ww1-new-2} For any N-function $\orf$ with $x^2\preccurlyeq\orf\preccurlyeq x^p$, $p>2$,
\begin{equation}\label{eq:w1-lp}
\ww_1(X_n,\sn)
\lesssim 
pL_\orf+\norm{V_n^2-1}_{1/2}^{1/2}.
\end{equation}
\end{enumerate}
As consequences,
 for any $p>2$,
\[
\ww_1(X_n,\sn)
\lesssim 
\min\left\{p,\, \log(e+L_p^{-2})\right\}L_p
+\norm{V_n^2-1}_{1/2}^{1/2}.
\]
In particular, for the $\ms L^\infty$ case, by \eqref{eq:def-y-infty},
\[
\ww_1(X_n,\sn)\lesssim 
L_\infty\log\left(e+L_\infty^{-2}\right)
+\norm{V_n^2-1}_{1/2}^{1/2}.
\]
\end{theorem}


\begin{remark}
\begin{enumerate}[(a)]
\item 
For any sup-quadratic N-function $\orf\succcurlyeq x^2$, by Lemma~\ref{alem:norm-dominate-ineq}, 
\[
s_n=\norm{Y}_2n^{1/2}\lesssim_{\orf(1)} \onorm{Y}n^{1/2}.
\] 
Hence, by the definition of $L_\orf$ in \eqref{def:L-orf}, 
\[
L_\orf\gtrsim_{\orf(1)} \orf^{-1}(n)n^{-1/2}
\gtrsim_{\orf(1)} \orf^{-1}(1)n^{-1/2}.
\]
Thus,   $\log(e+L_\orf^{-2})$ in Theorem~\ref{thm:ww1-new}\eqref{item:thm-ww1-new-1} is dominated by $\log(n)$.

\item Both bounds \eqref{eq:w1-general} and \eqref{eq:w1-lp} in Theorem~\ref{thm:ww1-new} have their strengths and weaknesses.

Apparently, the second bound \eqref{eq:w1-lp} gives better rates for mds which are at most polynomially integrable. E.g., for the ``barely more than square integrable" case $\orf=x^2\log(x+1)$, \eqref{eq:w1-lp} yields   $\frac{1}{s_n}(\frac{n}{\log n})^{1/2}(\log\log n)\onorm{Y}$ as the first term in the bound. However, \eqref{eq:w1-lp} becomes a trivial bound for $p\gg n^{1/2}$.

Although the first bound \eqref{eq:w1-general} is seemingly $\log(n)$ factor worse than the latter, 
it is applicable to martingales with more general integrability. For instance, 
if the martingale is exponentially integrable, taking $\orf=e^x-x-1$, Theorem~\ref{thm:ww1-new}\eqref{item:thm-ww1-new-1} yields  $\ww_1(X_n,\sn)\lesssim \onorm{Y}(\log n)^2/s_n+\norm{V_n^2-1}_{1/2}^{1/2}$, where the first term is better than the $O(n^{1/p}/s_n)$ rates offered by \eqref{eq:w1-lp} for any $p>0$.

\item 
In some sense, the term $\norm{V_n^2-1}$ quantifies the extent of decorrelation of the process. For independent sequences, $(V_n^2-1)$ is $0$.
In general, $(V_n^2-1)$ could converge to $0$ arbitrarily fast, depending on how decorrelated the process is. 

\item The term $(V_n^2-1)$ is essential for bounds of type 
\begin{equation}
\label{eq:w1-type}
\ww_r(X_n,\sn)\le C_1 L_\orf^a+C_2\norm{V_n^2-1}_c^b,
\end{equation}
i.e., we cannot allow $C_2$ to be $0$.
For example, let $B$ be such that $P(B=1\pm\tfrac{1}{2})=\tfrac{1}{2}$ and let $(\xi_i)_{i=1}^n$ be i.i.d. $\sn(0,1)$ variables. Define $Y_i=(B/n)^{1/2}\xi_i$. Then  $\lim_{n\to\infty}\ww_r(\sum_{i=1}^nY_i,\sn)=\ww_r(\sqrt{B}\xi_1,\sn)\neq 0$. Whereas, for $p> 2$, $L_p=O(n^{1/p-1/2})\to 0$ as $n\to\infty$. Thus $C_2\neq 0$.
\end{enumerate}
\end{remark}


Our next main results concern the $\ww_2$ and $\ww_3$ convergence rates.


\begin{theorem}[$\ww_2$ bounds]
\label{thm:ww2}
Let $\{Y_m\}_{m=1}^n$ be a martingale difference sequence. 
\begin{enumerate}[(i)]
\item\label{item:thm-ww2-1} 
For any N-function $\orf\preccurlyeq x^3$ such that $x\mapsto\orf(\sqrt{x})$ is convex,
\[
\ww_2(X_n,\sn)
\lesssim 
L_\orf+
\norm{V_n^2-1}_1^{1/2}.
\]

\item\label{item:thm-ww2-2} 
For any N-function $\orf\succcurlyeq x^3$,
\[
\ww_2(X_n,\sn)
\lesssim_{\orf(1)}
L_\orf+
\frac{\onorm{Y}}{s_n}\left[ng^{-1}(\orf^{-1}(n)^2)\right]^{1/4}
+\norm{V_n^2-1}_1^{1/2},
\]
where $g$ denotes the inverse function of $x\mapsto \tfrac{\orf(x)}{x}$.
\end{enumerate}
As consequences,
for $p\in(2,\infty)$,
\begin{align*}
\ww_2(X_n,\sn)
\lesssim 
\left\{
\begin{array}{lr}
L_p+\norm{V_n^2-1}_1^{1/2} &\text{ when }p\in(2,3],\\
n^{1/4+1/(2p^2-2p)}s_n^{-1}\norm{Y}_p+\norm{V_n^2-1}_1^{1/2}&\text{ when }p>3.
\end{array}
\right.
\end{align*}	
In particular, taking $p\to\infty$,
\begin{equation}\label{eq:ww2-L-infty}
\ww_2(X_n,\sn)
\lesssim 
\frac{n^{1/4}}{s_n}\norm{Y}_\infty+\norm{V_n^2-1}_1^{1/2}.
\end{equation}
\end{theorem}

\begin{theorem}[$\ww_3$ bound]
\label{thm:ww3}
Let $\{Y_m\}_{m=1}^n$ be a martingale difference sequence. Then
\[
\ww_3(X_n,\sn)\lesssim
L_3+\norm{V_n^2-1}_{3/2}^{1/2}.
\]
\end{theorem}

\begin{remark}
For the ease of our presentation, we only present results in terms of the $\ww_r$ metrics, $r\in\{1,2,3\}$. Readers can use the interpolation inequality
\[
\ww_r(X_n,\sn)\le \ww_j(X_n,\sn)^{j(k-r)/(k-j)}\ww_{k}(X_n,\sn)^{k(r-j)/(k-j)}, \quad j<r<k,
\]
 to easily infer other $\ww_r$ bounds, $r\in(1,3)$.
\end{remark}

The following proposition justifies that the terms $L_\orf$ and the exponent $1/2$ of  $\norm{V_n^2-1}_{1/2}^{1/2}$ within the $\ww_1$ bound \eqref{eq:w1-lp} are optimal.
\begin{proposition}\label{prop:optimal}(Optimality of the $\ww_1$ rates).
\begin{enumerate}[(1)]
\item\label{item:prop-opt-1} 
For any $p>2$, there exists a constant $N$ depending on $p$ such that
for any $n>N$,  we can find a martingale $\{X_i\}_{i=1}^n$ which satisfies
\begin{itemize}
\item  $V_n^2=1$, a.s.;
\item $L_p\lesssim_p(\log n)^{-1}$;
\item $ \ww_1(X_n,\sn)\asymp_p L_p$.
\end{itemize}
\item\label{item:prop-opt-2}  For any $n\ge 2$, there exists a martingale $\{X_i\}_{i=1}^n$ that satisfies
\begin{itemize}
\item $\norm{V_n^2-1}_{1/2}\lesssim (\log n)^{-2}$;
\item $\ww_1(X_n,\sn)
\asymp\norm{V_n^2-1}_{1/2}^{1/2}$.
\end{itemize}
\end{enumerate}
\end{proposition}

\section{Martingales with $V_n^2=1$}\label{sec:special}

In this section we will derive $\ww_r$ bounds, $r\in\{1,2,3\}$, for martingales with $V_n^2=1$. As in \cite{Rollin-18, FM-20, DMR-22,FS-22}, we will use Lindeberg's  argument and Taylor expansion to bound the Wasserstein distance with a sum involving the third order conditional moments. Our derivation of the $\ww_r$ bounds for $r>1$ also relies on 
an observation of Rio \cite{Rio-09} that relates the Wasserstein distances to  Zolotarev's ideal metric \eqref{eq:rio}. 
 
Note that
\cite{Rollin-18, FM-20, FS-22} use Stein's method to express the $\ww_1$ bound  in terms of derivatives of Stein's equation. But as observed in \cite{DMR-22},   Gaussian smoothing can already provide us enough regularity to do Taylor expansions. We follow \cite{DMR-22} in this respect.

\subsection{ Wasserstein distance bounds via Lindeberg's argument}


For any $n$-th order differentiable function $f$ and $\gamma\in(0,1]$,  denote its $(n,\gamma)$-H\"older constant by 
\[
[f]_{n,\gamma}:=\sup_{x,y:x\neq y}\frac{f^{(n)}(x)-f^{(n)}(y)}{|x-y|^\gamma}.
\]
We say that $f\in C^{n,\gamma}$ if $[f]_{n,\gamma}<\infty$.

\begin{definition}
Suppose $r>1$, and $\mu$, $\nu$ are probability measures on $\R$.
Let $\ell\in\N$ and $\gamma\in(0,1]$ be the unique numbers such that $r=\ell+\gamma$.
We let $\Lambda_r=\{f\in C^{\ell,\gamma}: [f]_{\ell,\gamma}\le 1\}$.  The {\it Zolotarev distance} $\zz_r$  is defined by 
\[
\zz_r(\mu,\nu)=\sup\left\{\int_{\R}f\dd\mu-\int_{\R}f\dd\nu: f\in\Lambda_r\right\}.
\]
\end{definition}
Recall that by the Kantorovich-Rubinstein Theorem, $\ww_r(\mu,\nu)=\zz_r(\mu,\nu)$ for $r=1$.
When $r>1$,  it is shown by Rio \cite[Theorem~3.1]{Rio-09} that 
\begin{equation}\label{eq:rio}
\ww_r(\mu,\nu)\lesssim_r\zz_r(\mu,\nu)^{1/r}.
\end{equation}

For $\sigma>0$ and any function $f\in\Lambda_r$, let 
\begin{equation}\label{eq:stein}
f_\sigma(x):=\ee{f(x+\sigma\sn)},
\end{equation}
where $\sn$ denotes a standard normal random variable. Direct integration by parts would yield the following regularity estimate for the Gaussian smoothing $f_\sigma$ 
which is a special case of \cite[Lemma 6.1]{DMR-09}.
\begin{lemma}
\label{lem:smoothing}\cite[Lemma 6.1]{DMR-09}
For $\sigma>0$ and $f\in\Lambda_r$, $r\in\N$, let $f_\sigma$ be as in \eqref{eq:stein}. Then
\[
\norm{f_\sigma^{(k)}}_\infty\le C_{k+1-r}\sigma^{r-k}, \quad\forall k\ge r,
\]
where $C_n=\int_\R \abs{u\phi^{(n)}(u)}\dd u$, and $\phi(x)=\tfrac{1}{\sqrt{2\pi}}e^{-x^2/2}$ is the standard normal density.
\end{lemma}
For the reader's convenience, we include a proof of Lemma~\ref{lem:smoothing} in the Appendix.

Recall the notations $\sigma_m(Y),X_m$ in \eqref{def:notations}.  For any random variable $U$, we write the expectation conditioning on $\ff_{i-1}$ as $E_i$,  i.e.,
\begin{equation}\label{eq:def-ei}
E_i[U]:=E[U|\ff_{i-1}].
\end{equation}
\begin{proposition}
\label{prop:v=1}
Let $Y=\{Y_m\}_{m=1}^n$, $n\ge 2$, be a martingale difference sequence and let $\tbd>0$.
Assume that $\sum_{i=1}^n\sigma_i^2(Y)=s_n^2$ almost surely.  
For any $\tbd>0$,  set
\[\lambda_m^2=\lambda_m^2(Y,\beta):=\sum_{i=m+1}^n\sigma_i^2(Y)+\tbd^2s_n^2.\]
Then, for any  $x^3\succcurlyeq\orf\in[0,\infty)^{[0,\infty)}$ such that $x\mapsto\orf(\sqrt{x})$ is convex, 
\begin{align}
\zz_r(X_n,\sn)
&\lesssim
\tbd^r+\frac{1}{s_n^r} E\sum_{i=1}^n\min\left\{\frac{E_i[|Y_i|^3]}{\lambda_i^{3-r}}, \frac{\lambda_i^rE_i\orf(|Y_i|)}{\orf(\lambda_i)}\right\},\, r\in\{1,2\},\label{eq:w12-prop}\\
\zz_3(X_n,\sn)
&\lesssim
\frac{1}{s_n^3}\sum_{i=1}^n\ee{|Y_i|^3}. \label{eq:w3-prop}
\end{align}	
\end{proposition}

\begin{proof}
[Proof of Proposition~\ref{prop:v=1}:]
Without loss of generality , assume $s_n=1$.

Let $\xi, \sn$ be random variables with distributions $\sn(0,\tbd^2)$ and $\sn(0,1)$, so that the triple $(\xi, \sn, \{X_i\}_{i=1}^n)$ are independent. For $h\in\Lambda_r$, $r\in\{1,2\}$, 
and any random variable $U$ which is independent of $\xi$,
\begin{align*}
\abs{\ee{h(U+\xi)-h(U)}}
=\Abs{\ee{h(U+\xi)-\sum_{i=0}^{r-1}h^{(i)}(U)\xi^i}}
\le 
\tfrac{1}{r!}\ee{|\xi|^r}.
\end{align*}	
The triangle inequality and \eqref{eq:rio} then yields, for $r\in\{1,2\}$,
\begin{equation}\label{eq:ww-elongate-err}
\abs{\zz_r(X_n,\sn)-
\zz_r(X_n+\xi,\sn+\xi)}\lesssim\tbd^r.
\end{equation}

In what follows we will use Lindeberg's argument to derive bounds for $\zz_r(X_n+\xi,\sn+\xi)$, $r\in\{1,2,3\}$. Recall notations in \eqref{def:notations}, \eqref{eq:def-ei}.

Note that $\lambda_m$ is $\ff_{m-1}$-measurable.  Recall the function $h_{\sigma}$ in \eqref{eq:stein}. Recall the operation $E_i$ in \eqref{eq:def-ei}. 
For $h\in \Lambda_r$, we consider the following telescopic sum:
\begin{align*}
\ee{h(X_n+\xi)-h(\sn+\xi)}
&=\sum_{i=1}^n\ee{h(X_i+\lambda_i\sn)-h(X_{i-1}+\lambda_{i-1}\sn)}
=E\sum_{i=1}^n
D_i,
\end{align*}	
where $D_i:=
E_i[h_{\lambda_{i}}(X_i)-h_{\lambda_{i}}(X_{i-1}+\sigma_i\sn)]$. 

Note that
\begin{equation}\label{eq:di-bound}
D_i=
E_i
\int_{\sigma_i\sn}^{Y_i}\int_0^th''_{\lambda_i}(X_{i-1}+s)-h''_{\lambda_i}(X_{i-1})\dd s\dd t.
\end{equation}
Hence, for $h\in\Lambda_r$, $r\in\{1,2,3\}$, using the fact $\sigma_i^3\le E_i[\abs{Y_i}^3]$,  we have
\begin{equation}
\label{eq:240613-1}
|D_i|\lesssim E_i[|Y_i|^3]\norm{h_{\lambda_i}^{(3)}}_\infty
\by{Lemma~\ref{lem:smoothing}}\lesssim
 E_i[|Y_i|^3]\lambda_i^{r-3}.
\end{equation}

When $r=3$, notice that $\tbd$ is irrelevant . So, letting $\tbd\downarrow 0$  we immediately get \eqref{eq:w3-prop}. It remains to consider $r\in\{1,2\}$.

Using the fact that  for $f\in C^{1,1}$,
\begin{align*}
\Abs{\int_0^y\int_0^t f''(x+s)-f''(x)\dd s\dd t}\le 
\min\left\{\tfrac{1}{6}|y|^3\norm{f^{(3)}}_\infty,  \, y^2\norm{f''}_\infty\right\},
\end{align*}
we get,  for any event $A_i$ and any $h\in\Lambda_r$, $r\in\{1,2\}$,
\begin{align*}
\rn{I}:&=\Abs{E_i
\int_{0}^{Y_i}\int_0^th''_{\lambda_i}(X_{i-1}+s)-h''_{\lambda_i}(X_{i-1})\dd s\dd t
}\\
&\le 
E_i\left[\tfrac{1}{6}|Y_i|^3\norm{h_{\lambda_i}^{(3)}}_\infty\idkt{A_i^c}+ Y_i^2\norm{h_{\lambda_i}''}_\infty\idkt{A_i}\right]\\
&\by{Lemma~\ref{lem:smoothing}}\lesssim
E_i\left[|Y_i|^3\lambda_i^{r-3}\idkt{A_i^c}+ Y_i^2\lambda_i^{r-2}\idkt{A_i}\right],
\end{align*}
and similarly 
\begin{align*}
\rn{II}:&=\Abs{E_i
\int_{0}^{\sigma_i\sn}\int_0^th''_{\lambda_i}(X_{i-1}+s)-h''_{\lambda_i}(X_{i-1})\dd s\dd t
}
\\&
\lesssim
E_i\left[\sigma^3\lambda_i^{r-3}\idkt{A_i^c}+ \sigma_i^2\lambda_i^{r-2}\idkt{A_i}\right].
\end{align*}
Further,  since $\orf\preccurlyeq x^3$, we have $s^3/t^3\le \orf(s)/\orf(t)$ for $0<s\le t$. Hence
\[
|Y_i|^3\lambda_i^{r-3}\idkt{|Y_i|<\lambda_i}
\le 
\lambda_i^r\frac{\orf(|Y_i|)}{\orf(\lambda_i)}\idkt{|Y_i|<\lambda_i}.
\]
Since $\orf\succcurlyeq x^2$, we have $s^2/t^2\le \orf_1(s)/\orf_1(t)$ for $0<t\le s$. Thus
\[
Y_i^2\lambda_i^{r-2}\idkt{|Y_i|\ge\lambda_i}
\le 
\lambda_i^r\frac{\orf(|Y_i|)}{\orf(\lambda_i)}\idkt{|Y_i|\ge\lambda_i},
\]
and so, for $r\in\{1,2\}$,
\begin{align*}
\rn{I}\lesssim
E_i\left[|Y_i|^3\lambda_i^{r-3}\idkt{|Y_i|<\lambda_i}+Y_i^2\lambda_i^{r-2}\idkt{|Y_i|\ge\lambda_i}\right]
\lesssim\frac{\lambda_i^rE_i[\orf(|Y_i|)]}{\orf(\lambda_i)}.
\end{align*}
Similarly, for $h\in\Lambda_r$, $r\in\{1,2\}$,
\[
\rn{II}
\lesssim 
E_i\left[\sigma^3\lambda_i^{r-3}\idkt{\sigma_i<\lambda_i}+ \sigma_i^2\lambda_i^{r-2}\idkt{\sigma_i\ge\lambda_i}\right]
\lesssim
\frac{\lambda_i^r\orf(\sigma_i)}{\orf(\lambda_i)}.
\]
Since $x\mapsto\orf(\sqrt{x})$ is convex, by Jensen's inequality we know that 
\begin{equation}\label{eq:orf-y-sigma}
\orf(\sigma_i)\le E_i[\orf(|Y_i|)].
\end{equation}
 Thus, we obtain,  for $h\in\Lambda_r$, $r\in\{1,2\}$,
\[
|D_i|\by{\eqref{eq:di-bound}}\le\rn{I}+\rn{II}\lesssim \frac{\lambda_i^rE_i[\orf(|Y_i|)]}{\orf(|\lambda_i|)}.
\]
This inequality, together with \eqref{eq:240613-1}, \eqref{eq:ww-elongate-err} and \eqref{eq:rio}, 
yields the Proposition.
\end{proof}

\begin{remark}
\label{rmk:ww-normal-role}
If there exists $1\le j\le n$ such that, given $\ff_{j-1}$, the distribution of $Y_{j}$  is the normal $\sn(0,\sigma_{j}^2)$, then the $j$-th summand in \eqref{eq:w12-prop} and \eqref{eq:w3-prop} can be removed. That is, the
summations in  \eqref{eq:w12-prop} and \eqref{eq:w3-prop} can both be replaced by
\[
\sum_{i=1,i\neq j}^n.
\]
Indeed, in this case, $X_{j}$ and $X_{j-1}+\sigma_j\sn$ are identically distributed conditioning on $\ff_{j-1}$. Thus $D_j=0$.
\end{remark}

\subsection{The rates of some special cases}
From Proposition~\ref{prop:v=1} we can obtain Wasserstein distance bounds for some special cases (with $V_n^2=1$).  See Corollary~\ref{cor:v=1} below. 
These cases usually yield better rates than the typical cases.
They are not only of interest on their own right, but can also serve as
important  references when we  construct counterexamples. 

The $\ww_1$ bounds within Corollary~\ref{cor:v=1} can be considered generalizations of
some of the results in \cite[Corollaries 2.2,2.3]{Rollin-18} and \cite[Corollaries 2.5,2.6,2.7]{FS-22} to the $\ms L^\orf$ integrable cases. 

\begin{corollary}
\label{cor:v=1}
Assume that $V_n^2=1$ almost surely.  Let $\orf$ be an N-function such that $\orf\preccurlyeq x^3$ and $x\mapsto\orf(\sqrt{x})$ is convex.  
Then the following statements hold.
\begin{enumerate}[(i)]
\item\label{item:cor-v1-mmt}
$\ww_r(X_n,\sn)\lesssim L_\orf$\, for $r\in\{1,2\}$.
\item \label{item:cor-v1-var-below}
If  there exists $\sigma>0$ such that 
\[
P(\sigma_i\ge\sigma\,\text{ for all }1\le i\le n)=1,
\] 
then, writing $M_\orf=\max_{i=1}^n\ee{\orf(|Y_i|)}$ (and write $M_\orf$ as $M_3$ when $\orf=x^3$),
\begin{itemize}
\item 
$
\ww_1(X_n,\sn)\lesssim
\left\{
\begin{array}{lr}
\frac{M_3}{\sigma^2 s_n}\log n &\text{ when }\orf=x^3,\\
(3-p)^{-1}\frac{M_\orf s_n^2}{\sigma^2\orf(s_n)}&\text{ when }\orf\preccurlyeq x^p, p\in(2,3);
\end{array}
\right.
$.
\item   $\ww_2(X_n,\sn)\lesssim \frac{M_\orf^{1/2} s_n}{\sigma\orf(s_n)^{1/2}}$.
\end{itemize}

 \item\label{item:cor-v1-bdd} If there exists a constant $\theta>0$ such that  almost surely,
 \[
E[\orf(|Y_i|)|\ff_{i-1}]\le\theta\sigma_i^2 \quad\forall i=1,\ldots,n, 
 \]
 then 
 \begin{itemize}
 \item 
 $
\ww_1(X_n,\sn)\lesssim_\theta 
\left\{
\begin{array}{lr}
\frac{1}{s_n}\log(e+s_n) &\text{ when }\orf=x^3,\\
\frac{s_n^2}{\orf(s_n)} &\text{ when }\orf\preccurlyeq x^p, p\in(2,3);
\end{array}
\right.
 $
\item $\ww_2(X_n,\sn)\lesssim_\theta \frac{s_n}{\orf(s_n)^{1/2}}$. 
\end{itemize}
In particular,  when $\theta=\max_{i=1}^n\norm{Y_i}_\infty<\infty$,  we have
\[
\ww_1(X_n,\sn)\lesssim \frac{\theta}{s_n}\log(e+s_n) \, \text{ and }\, 
\ww_2(X_n,\sn)\lesssim (\tfrac{\theta}{s_n})^{1/2}.
\]
\end{enumerate}
\end{corollary}

\begin{proof}
{\bf \eqref{item:cor-v1-mmt}}
 For the ease of notation we write $\theta:=\onorm{Y}$. 
 Note that $x\mapsto\orf(x/\theta)$ is still an N-function that satisfies conditions of Proposition~\ref{prop:v=1}. Recall the conditional expectation $E_i$ in \eqref{eq:def-ei}. 
By Proposition~\ref{prop:v=1},
\begin{align*}
\zz_r(X_n,\sn)
&\lesssim
\tbd^r+\frac{1}{s_n^r}E\sum_{i=1}^n\frac{\lambda_i^r E_i[\orf(|Y_i|/\theta)]}{\orf(\lambda_i/\theta)}\\
&\lesssim 
\tbd^r+\frac{1}{s_n^r}E\sum_{i=1}^n\frac{(\tbd s_n)^r E_i[\orf(|Y_i|/\theta)]}{\orf(\tbd s_n/\theta)}\\
&\lesssim 
\tbd^r+\frac{n\tbd^r}{\orf(\tbd s_n/\theta)} \quad \text{ for any }\tbd>0, r\in\{1,2\},
\end{align*}
where in the second inequality we used the fact that $\lambda_i\ge \tbd s_n$ for all $i$ and that $x\mapsto\frac{x^r}{\orf(x)}$ is decreasing. Taking $\tbd=\frac{\theta}{s_n}\orf^{-1}(n)$,
the  $\zz_1$, $\zz_2$ bounds are proved.

%
\medskip

\eqref{item:cor-v1-var-below} Recall $\lambda_i$ in Proposition~\ref{prop:v=1}.   When $\sigma_i\ge\sigma$ for all $i$ we have
\[
\lambda_i^2\ge (n-i)\sigma^2+\tbd^2s_n^2,  \quad\forall 1\le i\le n.
\]
Thus,  by Proposition~\ref{prop:v=1},  we get, for any $\tbd\ge 2\sigma/s_n$, $x^r\preccurlyeq \orf\preccurlyeq x^3$, $r\in\{1,2\}$,
\begin{align}\label{eq:0614-1}
\zz_r(X_n,\sn)
&\lesssim 
\tbd^r+\frac{M_\orf}{s_n^r}\sum_{i=1}^n\tfrac{[(n-i)\sigma^2+\tbd^2s_n^2]^{r/2}}{\orf(\sqrt{(n-i)\sigma^2+\tbd^2s_n^2})}\nn\\
&\lesssim 
\tbd^r+\frac{M_\orf}{s_n^r\sigma^2}
\int_{\sigma}^{\sqrt{n}\sigma+\tbd s_n}\frac{t^{r+1}}{\orf(t)}\dd t.
\end{align}
The first $\zz_1$ bound follows when $r=1$, $\orf(x)=x^3$,  $\tbd=2M_3/(s_n\sigma^2)$. 

Now consider the case $x^r\preccurlyeq \orf\preccurlyeq x^p$, $p\in(2,r+2)\cap(2,3]$, $r\in\{1,2\}$.  Since  $s_n^2\ge n\sigma$ and $x\mapsto x^p/\orf(x)$ is increasing, we get, for any $\tbd\in[\tfrac{2\sigma}{s_n},1]$,
\begin{align*}
\int_{\sigma}^{\sqrt{n}\sigma+\tbd s_n}\frac{t^p}{\orf(t)}t^{r+1-p}\dd t
&\le \frac{(2s_n)^p}{\orf(2s_n)}\int_{0}^{2s_n}t^{r+1-p}\dd t\\
&\le 
\frac{1}{r+2-p}\frac{(2s_n)^{r+2}}{\orf(2s_n)}\\
&\le 
\frac{1}{r+2-p}\frac{2^rs_n^{r+2}}{\orf(s_n)}.
\end{align*}
Taking $\tbd$ such that $\tbd^r=C\frac{M_\orf}{\sigma^2}\frac{s_n^2}{\orf(s_n)}$ for appropriate constant $C>0$ and recalling \eqref{eq:0614-1}, 
we get (the case $x^r\preccurlyeq \orf\preccurlyeq x^p$, $p\in(2,r+2)\cap(2,3]$)
\[
\zz_r(X_n,\sn)\lesssim\frac{1}{r+2-p}\tbd^r.
\]
the rest of the $\ww_r$ bounds in \eqref{item:cor-v1-var-below} follow.  Note that this is a trivial inequality (since $\ww_r(X_n,\sn)\lesssim 1$) if $\tbd>1$. If $\tbd<1$, it remains to justify that such a choice of $\tbd$ satisfies $\tbd\ge 2\sigma/s_n$. Indeed, 
since $\orf\succcurlyeq x^2$ and $\sigma_i\ge\sigma$, we have, for $1\le i\le n$,
\[
M_\orf\ge \frac{\orf(\sigma/2)}{(\sigma/2)^2}\ee{Y_i^2\idkt{|Y_i|\ge\sigma/2}}
\gtrsim \orf(\sigma/2)
\]
and so, for $r\in\{1,2\}$,
\[
\tbd s_n\ge \tbd^r s_n
\gtrsim 
\frac{\orf(\sigma/2) s_n^3}{(\sigma/2)^2\orf(s_n)}
\gtrsim \sigma,
\]
where in the last inequality we used the fact that $\orf\lesssim x^3$.
%
%

\medskip

\eqref{item:cor-v1-bdd} First, notice that there exists $\gamma>0$ such that $\sigma_i<\gamma$, $\forall 1\le i\le n$. Indeed, 
\begin{align*}
\frac{3}{4}\sigma_i^2
&\le E_i[Y_i^2\idkt{|Y_i|>\sigma_i/2}]
\le 
E_i[\orf(|Y_i|)]/g(\tfrac{\sigma_i}{2})
\le \theta\sigma_i^2/ g(\tfrac{\sigma_i}{2})
\end{align*}
where $g(x):=\orf(x)/x^2$ is an increasing function on $(0,\infty)$.  Thus 
\[
\sigma_i\le 2g^{-1}(\tfrac{4\theta}{3})=:\gamma.
\]

Next, by Proposition~\ref{prop:v=1} and $E_i[\orf(|Y_i|)]\le\theta\sigma_i^2$,  we get
\[
\zz_r(X_n,\sn)
\lesssim 
\tbd^r+\frac{\theta}{s_n^r}E\sum_{i=1}^n\frac{\lambda_i^r\sigma_i^2}{\orf(\lambda_i)}.
\]
 Define a sequence of stopping times $\{\tau_t\}_{0\le t\le1}$ by
\begin{equation}\label{def:stoppingt}
\tau(t)=\tau_t:=\inf\{m\ge 1: \sum_{i=1}^m\sigma_i^2\ge t\}.
\end{equation}

Clearly, $\tau_t=m$ if and only if $t\in(\sum_{i=1}^{m-1}\sigma_i^2,\sum_{i=1}^m\sigma_i^2]$.  
Moreover,
\[
\lambda_{\tau_t}^2=s_n^2(1+\tbd^2)-\sum_{i=1}^{\tau_t}\sigma_i^2\ge s_n^2(1+\tbd^2)-t-\sigma_{\tau_t}^2\ge  s_n^2(1+\tbd^2)-t-\gamma^2.
\] Thus we get,  for any $\tbd\in[\frac{\gamma}{s_n},1]$, $r\in\{1,2\}$,
\begin{align*}
\sum_{i=1}^n\frac{\lambda_i^r\sigma_i^2}{\orf(\lambda_i)}
=\int_0^{s_n^2}\frac{\lambda_{\tau_t}^r}{\orf(\lambda_{\tau_t})}\dd t\le 
\int_{s_n^2\tbd^2-\gamma^2}^{(1+\tbd^2)s_n^2}\frac{t^{r/2}}{\orf(t^{1/2})}\dd t
\le 
2\int_{\sqrt{s_n^2\tbd^2-\gamma^2}}^{\sqrt 2s_n}\frac{t^{r+1}}{\orf(t)}\dd t,
\end{align*}
where in the first inequality we used the fact that $x\mapsto x/\orf(x)$ is decreasing.
When $\orf=x^3$, taking $\tbd=2\gamma/s_n$,  the desired $\ww_1$ bound follows.
For the case $\orf\preccurlyeq x^p$, $2<p<3$, the bound of this integral can be handled the same way as in \eqref{item:cor-v1-var-below}.  Then the $\ww_1$ bounds in \eqref{item:cor-v1-bdd} are proved.

It remains to prove the $\zz_2$ bound. By the inequality above,  for any $\tbd\in[\frac{\gamma}{s_n},1]$, 
\begin{align*}
\zz_2(X_n,\sn)
\lesssim
\tbd^2+\frac{\theta}{s_n^2}\int_{0}^{\sqrt 2s_n}\frac{t^{3}}{\orf(t)}\dd t
\lesssim
\tbd^2+\frac{\theta s_n^2}{\orf(s_n)},
\end{align*}
where in the last inequality we  used the fact that $x\mapsto x^3/\orf$ is increasing.
Taking $\tbd^2=\frac{s_n^2\orf(\gamma)}{\orf(s_n)\gamma^2}$,  the $\zz_2$ bound in \eqref{item:cor-v1-bdd} follows.
\end{proof}

\section{Proofs of the main $\ww_r$ bounds}\label{sec:pf-main}
This section is devoted to the  proofs of Theorems~\ref{thm:ww1-new}, \ref{thm:ww2}, and \ref{thm:ww3}.


Note that the Taylor expansion results (Proposition~\ref{prop:v=1}), which are in terms of the second and third conditional moments, suit exactly martingales  (with $V_n^2=1$) with integrability between $\ms L^2$ and $\ms L^3$, and so not surprisingly,  the $\ww_1,\ww_2$ bounds of  Corollary~\ref{cor:v=1} for $\ms L^p$ martingales,  $p\in(2,3]\cup\{\infty\}$,  followed quite easily from Proposition~\ref{prop:v=1}. 
However, to obtain Wasserstein bounds for martingales with more general integrability, we need new insights.

In Subsection~\ref{subsec:modify}, we will modify the mds into a bounded sequence with deterministic total conditional variance $V_n^2$. To this end, we first truncate the martingale into a bounded sequence with $V_n^2\le 1$, and then lengthen it to have $V_n^2=1$. Such tricks of elongation and truncation of martingales were already used by Bolthausen \cite{Bol-82} (the idea goes back to Dvoretzky \cite{Dvo-72}) and Haeusler \cite{Haeusler-88} in the study of the Kolmogorov distance of the martingale CLT. 
The main result of  Subsection~\ref{subsec:modify} is a control of the error due to this modification in terms of $L_\orf$ and $(V_n^2-1)$. See Proposition~\ref{prop:trunc}. As an application, we prove the $\ww_3$ bound in Theorem~\ref{thm:ww3}.

In Subsections~\ref{subsec:pf-w1} and \ref{subsec:pf-w2}, we prove the $\ww_1,\ww_2$ estimates in Theorems~\ref{thm:ww1-new} and \ref{thm:ww2} by bounding the corresponding metrics for the modified martingale. A crucial step of our method is to use Young's inequality {\it inside} the conditional expectations  to ``decouple" the Lyapunov coefficient $L_\orf$ and the conditional variances $\sigma_i^2$ from the summation within  Proposition~\ref{prop:v=1}. This will turn the $\ww_r$ bounds into an optimization problem over three parameters: the truncation parameter ($\alpha$), the smoothing parameter ($\tbd$), and an Orlicz parameter ($\scon$). To this end, tools from the theory of N-functions will be employed to compare N-functions to polynomials.

 \subsection{A modified martingale, and Proof of Theorem~\ref{thm:ww3} ($\ww_3$ bounds)}\label{subsec:modify}

The goal of this subsection is to modify the original mds $\{Y_i\}_{i=1}^n$ to a new mds $\hat{Y}=\{\hat Y_i\}_{i=1}^{n+1}$ which is uniformly bounded  except the last term. Throughout this subsection, we let $\alpha\in(0,\infty]$ be any fixed constant. 
 
 Define a martingale difference sequence $Z=Z^{(\alpha)}=\{Z_i\}_{i=1}^n$ as
\begin{equation}\label{eq:def-yhat-1}
Z_i=Z_i^{(\alpha)}:= Y_i\mathbbm{1}_{|Y_i|\le\alpha/2}-E[Y_i\mathbbm{1}_{|Y_i|\le\alpha/2}|\ff_{i-1}], \quad 1\le i\le n.
\end{equation}
Note that $P(|Z_i|\le\alpha)=1$ for all $1\le i\le n$ and, with $\sigma_i^2(Z):=\ee{Z_i^2|\ff_{i-1}}$,
\begin{align}\label{eq:var-hat}
\sum_{i=1}^n\sigma_i^2(Z)
&\le \sum_{i=1}^n\ee{Y_i^2\idkt{|Y_i|\le\alpha/2}}
\le\sum_{i=1}^n\sigma_i^2(Y).
\end{align}	
Define a stopping time (with the convention $\inf\emptyset=\infty$)
\begin{equation}\label{eq:def-st}
T=\inf\big\{1\le m\le n:\sum_{i=1}^m\sigma_i^2(Z)>s_n^2\big\}\wedge(n+1).
\end{equation}
Since $\sum_{i=1}^m\sigma_i^2(Z)$ is $\ff_{m-1}$-measurable, 
 $(T-1)$ is also a stopping time.
 
\begin{definition}\label{def:trunc-mart}
Let $\alpha\in(0,\infty]$ and let $\xi\sim\sn(0,1)$ be a standard normal  which is independent of $\{Y_i\}_{i=1}^n$.
Define the modified martingale difference sequence  $\hat{Y}=\hat{Y}^{(\alpha)}=\{\hat Y_i\}_{i=1}^\infty$ as follows: $\forall i\ge 1$,
\begin{equation}\label{eq:def-yhat-2}
\hat{Y}_i=\hat{Y}_i^{(\alpha)}:=
Z_i\idkt{\sum_{j=1}^i\sigma_j^2(Z)\le 1}
+\xi\left(s_n^2-\sum_{j=1}^{T-1}\sigma_j^2(Z)\right)^{1/2}\idkt{i=n+1}.
\end{equation}
Enlarging the $\sigma$-fields $\{\ff_i\}$ if necessary,  $\{\hat{Y}_m\}_{m\ge 0}$ is still a mds. Let
\[
\hat{X}_m=\frac{1}{s_n}\sum_{i=1}^m\hat{Y}_i, \quad m\ge 1, \text{ with }\hat{X}_0=0.
\] 
\end{definition}

Note that  $\hat{Y}_j=0$ for all $j\in[T,n]\cup(n+1,\infty)$ and recall that $T\le n+1$.  We write, for $m\ge 1$,  $\sigma_m^2(\hat{Y}):=\ee{\hat{Y}_m^2|\ff_{m-1}}$.  

The main feature of the mds $\hat{Y}$ is that $|\hat{Y}_i|\le\alpha$ a.s. for $1\le i\le n$, and 
\begin{equation}\label{eq:var-xhat}
\sum_{i=1}^{n+1}\sigma_i^2(\hat{Y})=s_n^2 \quad\text{ almost surely.}
\end{equation}

\begin{proposition}
\label{prop:trunc}
For $\alpha\in(0,\infty]$ and a mds $\{Y_i\}_{i=1}^n$, 
let $\{\hat{Y}_i\}$ and $Z=\{Z_i\}_{i=1}^n$ be as in Definition~\ref{def:trunc-mart}.  
Let $r\ge 1$ and let  $\orf\in[0,\infty)^{[0,\infty)}$ be an N-function with $\orf\succcurlyeq x^{r\vee 2}$. Recall the definition of $L_\orf=L_\orf(Y)$ in \eqref{def:L-orf}. 
Set $\fcon:=\onorm{Y}$.
\begin{enumerate}[(a)]
\item\label{item:prop-trunc-1} For $\alpha<\infty$, $r\ge 1$,
\begin{align*}
\abs{\ww_r(X_n,\sn)-\ww_r(\hat{X}_{n+1},\sn)}
\lesssim_r
L_\orf+\alpha
+\frac{n^{1/2}\alpha}{\orf(\alpha s_n/2\fcon)^{1/(r\vee 2)}}+\norm{V_n^2-1}_{r/2}^{1/2}.
\end{align*}
\item\label{item:prop-trunc-2} 
If $\orf$ satisfies that $x\mapsto\orf(\sqrt{x})$ is convex, 
then for $r\ge 1$,
\begin{align*}
\abs{\ww_r(X_n,\sn)-\ww_r(\hat{X}_{n+1},\sn)}
\lesssim_r 
L_\orf
+\frac{n^{1/2}\alpha}{\orf(\alpha s_n/2\fcon)^{1/(r\vee 2)}}\idkt{\alpha<\infty}+\norm{V_n^2-1}_{r/2}^{1/2}.
\end{align*}
\end{enumerate}
\end{proposition}
Estimate \eqref{item:prop-trunc-2} is slightly better than \eqref{item:prop-trunc-1} but with a slightly stronger condition.  Statement \eqref{item:prop-trunc-2} is more useful for the case $\alpha=\infty$.

\begin{lemma}
\label{lem:orfnorm-z}
Let $Z=\{Z_i\}_{i=1}^n$ be as in \eqref{eq:def-yhat-1}.
Recall the definition of the orlicz norm $\onorm{\cdot}$ for a sequence in Definition~\ref{def:orlicz-norm}.
 Then, for any N-function $\orf$, 
\[
\onorm{Z}\le 2\onorm{Y}
\]
\end{lemma}
\begin{proof}
Without loss of generality, assume $\onorm{Y}=1$.
It suffices to show that
\begin{equation}\label{eq:onorm-hatphi-phi}
E\sum_{i=1}^n\orf(|Z_i|/2)\le n.
\end{equation}
Indeed, recall  $E_i[\cdot]$ in \eqref{eq:def-ei}.  By the definition of $Z_i$ in \eqref{eq:def-yhat-1}, for $1\le i\le n$,
\begin{align*}
E_i[\orf(|Z_i|/2)]
&\le E_i\left[\orf\left(
(|Y_i|\idkt{|Y_i|\le\alpha/2}+E_i[|Y_i|\idkt{|Y_i|\le\alpha/2}])/2
\right)\right]\\
&\le 
\frac{1}{2}E_i\left[\orf(|Y_i|\idkt{|Y_i|\le\alpha/2})+\orf(E_i[|Y_i|\idkt{|Y_i|\le\alpha/2}])\right]\\
&\le 
E_i\left[\orf(|Y_i|\idkt{|Y_i|\le\alpha/2})\right],
\end{align*}
where we used Jensen's inequality in both the second and the third inequalities.
Summing both sides over all $1\le i\le n$,
inequality \eqref{eq:onorm-hatphi-phi} follows.
\end{proof}

\begin{proof}[Proof of Proposition~\ref{prop:trunc}]
Throughout, notice that $r\ge 1$ and $\orf\succcurlyeq x^{r\vee 2}$.

Without loss of generality, assume $s_n=1$. Set
\begin{align}
K_r:=\norm{V_n^2-1}_{r/2}^{1/2}. \label{eq:def-K}
\end{align}	

To prove Proposition~\ref{prop:trunc},
it suffices to show that
\begin{equation}\label{eq:l2-bd-diff-tr-elon}
\norm{X_n-\hat{X}_{n+1}}_r
\lesssim_r
L_\orf+\norm{\max_i^n\sigma_i(Z)}_r+K_r
+\frac{n^{1/2}\alpha}{\orf(\alpha/2\fcon)^{1/(r\vee 2)}}\idkt{\alpha<\infty}.
\end{equation}
Indeed, trivially we have $\norm{\max_i^n\sigma_i(Z)}_r\le\alpha$. If $\orf(\sqrt{x})$ is convex, we claim that
\begin{equation}\label{eq:ifconvex}
\norm{\max_i^n\sigma_i(Z)}_r\lesssim L_\orf.
\end{equation}
To this end, we apply Lemma~\ref{alem:square-by-orlicz} to get
\[
\norm{\max_i^n\sigma_i(Z)}_r
\le 2^{1/r}\onorm{\{\sigma_i(Z)\}_{i=1}^n}\orf^{-1}(n)
\]
Notice that, when $\orf(\sqrt{x})$ is convex,  then using Jensen's inequality as in \eqref{eq:orf-y-sigma}, we have
$\onorm{\{\sigma_i(Z)\}_{i=1}^n}\le\onorm{Z}$
which is bounded by $2\onorm{Y}$ by Lemma~\ref{lem:orfnorm-z}.  Claim \eqref{eq:ifconvex} follows.


The rest of the proof is devoted to obtain \eqref{eq:l2-bd-diff-tr-elon}.

Recall $T$ in \eqref{eq:def-st}. 
Since $\{\hat{Y}_i\}$ and $\{Z_i\}$ coincide up to $i=T-1$, we have
\begin{align}\label{eq:decompose-diff}
X_n-\hat{X}_{n+1}
&=\sum_{i=T}^{n}Z_i+\sum_{i=1}^n(Y_i-Z_i)-\hat{Y}_{n+1}.
\end{align}	

{\bf Step 1.} 
We will bound the first term $\sum_{i=T}^{n}Z_i$ in \eqref{eq:decompose-diff} by
\begin{equation}
\label{eq:newmg-term1}
\norm{\sum_{i=T}^n Z_i}_r
\lesssim_r
K_r+L_\orf, \quad r\ge 1.
\end{equation}
Since $T-1$ is a stopping time, $\{Z_i\}_{i=T}^n$ form a martingale difference sequence. Hence, by 
an inequality of Burkholder Theorem~\ref{athm:hall-heyde}
 and Lemma~\ref{alem:square-by-orlicz}, for $r\ge 1$,
\begin{align}\label{eq:norm1-z}
\norm{\sum_{i=T}^n Z_i}_r
\le 
\norm{\sum_{i=T+1}^n Z_i}_r+\norm{\max_{i=1}^n|Z_i|}_{r}
\lesssim_r
\norm{\sum_{i=T+1}^n\sigma_i^2(Z)}_{r/2}^{1/2}+L_\orf.
\end{align}
Further, by the definition of the stopping time $T$,
\begin{align*}
\sum_{i=T+1}^n\sigma_i^2(Z)
=
\left(\sum_{i=1}^n\sigma_i^2(Z)-\sum_{i=1}^T\sigma_i^2(Z)\right)\vee 0
\le 
\left(
\sum_{i=1}^n\sigma_i^2(Z)-1
\right)\vee 0.
\end{align*}
This inequality, together with \eqref{eq:norm1-z},
yields \eqref{eq:newmg-term1}.

{\bf Step 2.} 
Consider the second $\sum_{i=1}^n(Y_i-Z_i)$ in \eqref{eq:decompose-diff}. 
We will show that
\begin{equation}
\label{eq:Y-Z-bound}
\norm{\sum_{i=1}^n(Y_i-Z_i)}_r
\lesssim_r
\frac{\alpha n^{1/2}}{\orf(\alpha/2\fcon)^{1/(r\vee2)}}\idkt{\alpha<\infty}.
\end{equation}
Indeed, this inequality is trivial for $\alpha=\infty$.
When $\alpha<\infty$, note that $\{Y_i-Z_i\}_{i=1}^n$ is a martingale difference sequence.
For $r\ge 1$, by Burkholder's inequality, 
\begin{align*}
\norm{\sum_{i=1}^n(Y_i-Z_i)}_r
\lesssim_r
\norm{
\sum_{i=1}^n(Y_i-Z_i)^2}_{r/2}^{1/2}
\end{align*}
Further, by H\"older's inequality, Jensen's inequality, and the fact
$Y_i-Z_i= Y_i\mathbbm{1}_{|Y_i|>\alpha/2}-E[Y_i\mathbbm{1}_{|Y_i|>\alpha/2}|\ff_{i-1}]$,
 we get, for $r\ge 2$, 
\begin{align*}
\norm{
\sum_{i=1}^n(Y_i-Z_i)^2}_{r/2}
\le 
n\bnorm{\frac{1}{n}\sum_{i=1}^n|Y_i-Z_i|^r}_1^{2/r}\le
4n\bnorm{\frac{1}{n}\sum_{i=1}^n|Y_i|^r\idkt{|Y_i|>\alpha/2}}_1^{2/r}.
\end{align*}
Since $|Y_i|^r\idkt{|Y_i|>\alpha/2}
\le 
\orf(|Y_i|/\fcon)(\alpha/2)^r/\orf(\alpha/2\fcon)$, we have
\begin{equation}
\label{eq:noname}
\bnorm{\frac{1}{n}\sum_{i=1}^n|Y_i|^r\idkt{|Y_i|>\alpha/2}}_1
\le 
\frac{(\alpha/2)^r}{n\orf(\alpha/2\fcon)}\sum_{i=1}^n\ee{\orf(|Y_i|/2\fcon)}
\le \frac{(\alpha/2)^r}{\orf(\alpha/2\fcon)}
\end{equation}
where the last inequality used the definition of $\onorm{Y}$.
Hence
\[
\norm{
\sum_{i=1}^n(Y_i-Z_i)^2}_{r/2}
\le
\frac{\alpha^2n}{\orf(\alpha/2\fcon)^{2/(r\vee2)}}.
\]
Inequality \eqref{eq:Y-Z-bound} is proved.

\medskip
{\bf Step 3.} Consider the last term in \eqref{eq:decompose-diff}. By definition \eqref{eq:def-yhat-2}, 
\[
\norm{\hat{Y}_{n+1}}_r
\lesssim_r
\bnorm{\big(1-\sum_{i=1}^{T-1}\sigma_i^2(Z)\big)^{1/2}}_r.
\]

%

When $T\le n$, we have
\[
0\le 1-\sum_{i=1}^{T-1}\sigma_i^2(Z)\le 1-\sum_{i=1}^{T}\sigma_i^2(Z)+\max_{i=1}^n\sigma_i^2(Z)<\max_{i=1}^n\sigma_i^2(Z).
\]
If $T=n+1$,  then
\[
\Abs{1-\sum_{i=1}^{T-1}\sigma_i^2(Z)}
\le 
\Abs{1-\sum_{i=1}^{n}\sigma_i^2(Y)}
+
\Abs{\sum_{i=1}^{n}\sigma_i^2(Y)-\sum_{i=1}^{n}\sigma_i^2(Z)}.
\]
Hence, we have
\begin{equation}\label{eq:240530-1}
\norm{\hat{Y}_{n+1}}_r
\le 
\norm{\max_i^n\sigma_i(Z)}_r+K_r+
\bnorm{\sum_{i=1}^{n}\sigma_i^2(Y)-\sigma_i^2(Z)}_{r/2}^{1/2}.
\end{equation}

{\bf Step 4.} 
Let us bound the third term in \eqref{eq:240530-1}. We will show that
\begin{equation}
\label{eq:last-term-bd}
\bnorm{\sum_{i=1}^{n}\sigma_i^2(Y)-\sigma_i^2(Z)}_{r/2}
\lesssim
\frac{n\alpha^2}{\orf(\alpha/2\fcon)^{2/(r\vee 2)}}\idkt{\alpha<\infty}.
\end{equation}
This is trivial when $\alpha=\infty$. For $\alpha<\infty$,
recall the notation $E_i$ in \eqref{eq:def-ei}. Then
\begin{align*}
\sigma_i^2(Y)-\sigma_i^2(Z)
=E_i[Y_i^2\idkt{|Y_i|>\alpha/2}]+E_i[Y_i\idkt{|Y_i|>\alpha/2}]^2
\quad\forall 1\le i\le n.
\end{align*}	
Hence, by H\"older's inequality and Jensen's inequality, for $r\ge 2$, 
\begin{align*}
\bnorm{\sum_{i=1}^{n}\sigma_i^2(Y)-\sigma_i^2(Z)}_{r/2}
&\le
2\bnorm{\sum_{i=1}^{n}E_i[Y_i^2\idkt{|Y_i|>\alpha/2}]}_{r/2}\\
&\le
2n\bnorm{\frac{1}{n}\sum_{i=1}^n|Y_i|^r\idkt{|Y_i|>\alpha/2}}_1^{2/r}\by{\eqref{eq:noname}}\lesssim
\frac{n\alpha^2}{\orf(\alpha/2\fcon)^{2/r}}.
\end{align*}
Display \eqref{eq:last-term-bd} is proved.

{\bf Step 5.} 
By \eqref{eq:240530-1}, \eqref{eq:last-term-bd}, we have arrived at
\[
\norm{\hat{Y}_{n+1}}_r
\lesssim
\norm{\max_i^n\sigma_i(Z)}_r+K_r+\frac{\alpha n^{1/2}}{\orf(\alpha/2\fcon)^{1/(r\vee2)}}\idkt{\alpha<\infty}.
\]
This inequality, together with \eqref{eq:decompose-diff}, \eqref{eq:newmg-term1}, \eqref{eq:Y-Z-bound}, yields \eqref{eq:l2-bd-diff-tr-elon}.
%
%
\end{proof}

\begin{proof}
[Proof of Theorem~\ref{thm:ww3}]
Consider $\alpha=\infty$ and recall the definition of $\hat{Y}=\hat{Y}^{(\infty)}$ in Definition~\ref{def:trunc-mart}. Note that $Z_i^{(\infty)}=Y_i$ for $1\le i\le n$.

Since $\hat{Y}_{n+1}$ is normal conditioning on $\ff_n$,
by \eqref{eq:w3-prop} in  Proposition~\ref{prop:v=1}  and  Remark~\ref{rmk:ww-normal-role}, we have
\[
\zz_3(\hat{X}_{n+1},\sn)
\lesssim\frac{1}{s_n^3}\sum_{i=1}^n\ee{|\hat{Y}_i|^3}
\le\frac{1}{s_n^3}\sum_{i=1}^n\ee{|Z_i|^3}
=
\frac{1}{s_n^3}\sum_{i=1}^n\ee{|Y_i|^3}.
\]
By \eqref{eq:rio}, this implies $\ww_3(\hat{X}_{n+1},\sn)\lesssim L_3$. 
Further, applying  Proposition~\ref{prop:trunc}\eqref{item:prop-trunc-2} to the case $\alpha=\infty$, we get
\[
\ww_3(X_n,\sn)
\lesssim
\ww_3(\hat{X}_{n+1},\sn)+L_3+\norm{V_n^2-1}_{3/2}^{1/2}.
\]
Theorem~\ref{thm:ww3} is proved.
\end{proof}

\subsection{Proof of Theorem~\ref{thm:ww1-new} ($\ww_1$ bounds)}\label{subsec:pf-w1}


In what follows we let $\alpha\in(0,\infty),\scon>0$ be  constants to be determined later. Let the mds $\hat{Y}=\hat{Y}^{(\alpha)}$ be as in Definition~\ref{def:trunc-mart},  and set 
\begin{equation}\label{eq:def-kappa}
\beta=\sqrt{2}\alpha.
\end{equation}


Throughout this subsection, we simply write
\begin{align}
\sigma_i^2&=\sigma_i^2(\hat Y)=E[\hat Y_i^2|\ff_{i-1}], \quad i\ge 1\nn\\
\lambda_m^2&=\sum_{i=m+1}^{n+1}\sigma_i^2+\tbd^2 s_n^2, \quad 0\le m\le n+1.\label{eq:def-lambda}
\end{align}	
Note that $\lambda_0^2=(1+\tbd^2)s_n^2$, and $\lambda_m^2=\lambda_0^2-\sum_{i=1}^m\sigma_i^2$ is $\ff_{m-1}$-measurable for $m\ge 1$.

Recall that $\hat{Y}_{n+1}$ is normaly distributed given $\ff_{i-1}$. 
By \eqref{eq:var-xhat}, Proposition~\ref{prop:v=1}, and Remark~\ref{rmk:ww-normal-role},  for $r\in\{1,2\}$, $\alpha>0,\tbd=\sqrt{2}\alpha$,
\begin{equation}\label{eq:wr-elongated}
\ww_r(\hat{X}_{n+1},\sn)
\lesssim
\alpha^r+s_n^{-r}E\sum_{i=1}^{n}\frac{|\hat Y_i|^3}{\lambda_{i}^{3-r}}.
\end{equation}

The following estimate of the summation within \eqref{eq:wr-elongated} will be crucially employed in the derivation of our $\ww_r$ bounds, $r=1,2$. 

\begin{proposition}
\label{prop:young-decoup}
Let $\orf\in[0,\infty)^{[0,\infty)}$ be an N-function with $\orf\succcurlyeq x^2$, 
and set $\fcon=\onorm{Y}$. 
Recall $\alpha,\hat{Y}, \lambda_i$ above, and $\tbd=\sqrt{2}\alpha$. 
Then, for any $B, k>0$,
\begin{equation*}
E\sum_{i=1}^n\frac{|\hat{Y}_i|^3}{\lambda_i^k}\idkt{\lambda_i^2\le B}
\lesssim
\fcon\scon n+\fcon\scon\alpha^{-2}\int_{\alpha^2}^{(1+\alpha^2)\wedge B} \orf_*(\tfrac{\alpha^2}{\scon t^{k/2}})\dd t,\quad\forall \scon>0.
\end{equation*}
\end{proposition}
\begin{proof}
Without loss of generality, assume $s_n=1$. 

Recall that $\lambda_i$ is $\ff_{i-1}$-measurable. Recall that $E_i[\cdot]=\ee{\cdot|\ff_{i-1}}$ as in \eqref{eq:def-ei}.

By Young's inequality \eqref{eq:young},  for any $\scon>0$,
\[
E\sum_{i=1}^{n}
\frac{E_i[|\hat Y_i|^3]}{2\fcon\scon\lambda_i^k}\idkt{\lambda_i^2\le B}
\le 
E\sum_{i=1}^{n}E_i\big[\orf(\tfrac{|\hat{Y}_i|}{2\fcon})+\orf_*(\tfrac{\hat{Y}_i^2}{\scon\lambda_i^k})\idkt{\lambda_i^2\le B}\big].
\]
Note that $\orf_*$ is an N-function and so $\orf_*\succcurlyeq x$. Using the fact that $\orf_*\succcurlyeq x$ and that $|\hat{Y}_i|\le\alpha$ for $1\le i\le n$, we have, almost surely,
\[
\orf_*(\tfrac{\hat{Y}_i^2}{\scon\lambda_i^k})
\le \tfrac{\hat{Y_i}^2}{\alpha^2}\orf_*(\tfrac{\alpha^2}{\scon\lambda_i^k}),\quad\forall 1\le i\le n.
\]
Thus we further have (Note that $\hat{Y}_i=0$ for $T\le i\le n$.)
\begin{align*}
E\sum_{i=1}^{n}
\frac{E_i[|\hat Y_i|^3]}{2\fcon\scon\lambda_i^k}\idkt{\lambda_i^2\le B}
&\le  
E\sum_{i=1}^n\orf(\tfrac{|Z_i|}{2\fcon})+E\sum_{i=1}^{n}\tfrac{E_i[\hat{Y}_i^2]}{\alpha^2}\orf_*(\tfrac{\alpha^2}{\scon\lambda_i^k})\idkt{\lambda_i^2\le B}\\
&\by{Lemma~\ref{lem:orfnorm-z}}\le
n+E\sum_{i=1}^{T-1}\tfrac{\sigma_i^2}{\alpha^2}\orf_*(\tfrac{\alpha^2}{\scon\lambda_i^k})\idkt{\lambda_i^2\le B}\quad\forall\scon>0.
\end{align*}

Next, we will show the following integral bound:
\begin{align}
\label{eq:w1-integral}
E\sum_{i=1}^{T-1} \sigma_i^2\orf_*(\tfrac{\alpha^2}{\scon\lambda_{i-1}^k})\idkt{\lambda_i^k\le  B}
\le 
\int_{\alpha^2}^{(1+\alpha^2)\wedge B} \orf_*(\tfrac{\alpha^2}{\scon t^{k/2}})\dd t.
\end{align}
Recall the  sequence of stopping times $\{\tau_t\}_{0\le t\le1}$ defined in \eqref{def:stoppingt}.
Then $\tau_t=m$ if and only if $t\in(\sum_{i=1}^{m-1}\sigma_i^2,\sum_{i=1}^m\sigma_i^2]$.  
Moreover, recalling that $\tbd^2=2\alpha^2$,
\[
\lambda_{\tau_t}^2=1+\tbd^2-\sum_{i=1}^{\tau_t}\sigma_i^2\ge 1+\tbd^2-t-\alpha^2=1+\alpha^2-t.
\]
Thus,  for any $\alpha>0$, writing $v_m^2:=\sum_{i=1}^m\sigma_i^2$,
\begin{align*}
\bee{\sum_{i=1}^{T-1} \sigma_i^2\orf_*(\tfrac{\alpha^k}{\scon\lambda_{i}^k})\idkt{\lambda_i^2\le  B}}
&=E\sum_{i=1}^{T-1}\int_{v_{i-1}^2}^{v_i^2}\orf_*(\tfrac{\alpha^2}{\scon\lambda_{\tau_t}^k})\idkt{\lambda_{\tau_t}^2\le B}\dd t\\
&\le E\int_0^1 \orf_*(\tfrac{\alpha^2}{\scon(1+\alpha^2-t)^{k/2}})\idkt{1+\alpha^2-t\le  B}\dd t\\
&\le 
\int_{\alpha^2}^{(1+\alpha^2)\wedge B} \orf_*(\tfrac{\alpha^2}{\scon t^{k/2}})\dd t.
\end{align*}
Inequality \eqref{eq:w1-integral} is proved.  The Proposition follows.
\end{proof}

To prove Theorem~\ref{thm:ww2} and Theorem~\ref{thm:ww3}, we will need the following lemma to compare N-functions to power functions.
\begin{lemma}\label{lem:succsim-domination}
Let $p>1$ and let $q:=p/(p-1)$ denote its H\"older conjugate.  Let $\orf$ be an N-function.
Set $C_{\orf}=\max_{x,y>0}\frac{\orf^{-1}(x)\orf_*^{-1}(x)/x}{\orf^{-1}(y)\orf_*^{-1}(y)/y}$.
\begin{enumerate}[(1)]
\item\label{item:lem-succ} If  $\orf\preccurlyeq x^p$, then
\[
\frac{\orf_*(s)}{s^q}\ge C_{\psi}^{-q}\frac{\orf_*(t)}{t^q}
\quad\text{ for any }s>t>0.
\]
\item\label{item:lem-prec} If $\orf\succcurlyeq x^p$, then
\[
\frac{\orf_*(s)}{s^q}\le C_{\psi}^{q}\frac{\orf_*(t)}{t^q}
\quad\text{ for any }s>t>0.
\]
\end{enumerate}
Note that $C_{\orf}=1$ if $\orf=x^p$,  and \eqref{eq:orlicz-ineq} guarantees that $1\le C_\orf\le 2$ in general.
\end{lemma}

\begin{proof}[Proof of Theorem~\ref{thm:ww1-new}:]
Without loss of generality, assume $s_n=1$.  Set $\fcon=\onorm{Y}$. 

\noindent{\bf Part 1 (Proof of Theorem~\ref{thm:ww1-new}\eqref{item:thm-ww1-new-1}
 for general $\orf$).}

Since $\orf_*(x)/x$ is increasing,  we have $t\orf_*(\tfrac{\alpha^2}{\scon t})\le\alpha^2\orf_*(\tfrac{\alpha^2}{\scon\alpha^2})$ for $t\ge \alpha^2$.  Hence
\begin{align*}
\int_{\alpha^2}^{1+\alpha^2} \orf_*(\tfrac{\alpha^2}{\scon t})\dd t
\le 
\int_{\alpha^2}^{1+\alpha^2}\frac{\alpha^2}{t} \orf_*(\tfrac{1}{\scon})\dd t
= \alpha^2\orf_*(\tfrac{1}{\scon})\log(1+\tfrac{1}{\alpha^2}).
\end{align*}
This inequality, together with  Proposition~\ref{prop:young-decoup}, yields
\begin{align*}  
E\sum_{i=1}^{n} \frac{|\hat{Y}_i|^3}{\lambda_{i-1}^2}
&\lesssim 
\fcon\scon n+\fcon\scon\orf_*(\tfrac{1}{\scon})\log(1+\tfrac{1}{\alpha^2}).
\end{align*}
Taking $\scon=1\big/\orf_*^{-1}(n)$ in the above inequality, we obtain
\begin{align}
\label{eq:240523-1}
E\sum_{i=1}^{T-1} \frac{|\hat{Y}_i|^3}{\lambda_{i-1}^2}\lesssim
\fcon\scon n \log(e+\tfrac{1}{\alpha^2})
\lesssim
\fcon\orf^{-1}(n)\log(e+\tfrac{1}{\alpha^2})
\end{align}
where in the last inequality we used the fact 
\begin{equation}\label{eq:rho-upper}
\scon=1\big/\orf_*^{-1}(n)
\by{\eqref{eq:orlicz-ineq}}\le\tfrac{1}{n}\orf^{-1}(n).
\end{equation}
Further, putting 
\begin{equation}\label{eq:alpha-value}
\alpha=2\fcon\orf^{-1}(n)=2L_\orf,
\end{equation}
then inequality \eqref{eq:240523-1} and \eqref{eq:wr-elongated} imply  
\[
\ww_1(\hat{X}_{n+1},\sn)
\lesssim\alpha\log(e+\tfrac{1}{L_\orf^2}). 
\]
This inequality, together with Proposition~\ref{prop:trunc}\eqref{item:prop-trunc-1}, yields (recall $K_1$ in \eqref{eq:def-K})
\begin{align*}
\ww_1(X_n,\sn)
\lesssim 
\alpha\log(e+\tfrac{1}{L_\orf^2})+\frac{n^{1/2}\alpha}{\orf(\alpha/2\fcon)^{1/2}}+K_1
\by{\eqref{eq:alpha-value}}\lesssim
\alpha\log(e+\tfrac{1}{L_\orf^2})+K_1.
\end{align*}	
 Theorem~\ref{thm:ww1-new}\eqref{item:thm-ww1-new-1} is proved.

{\bf Part 2 (Proof of
 Theorem~\ref{thm:ww1-new}\eqref{item:thm-ww1-new-2} for $\orf\preccurlyeq x^p$, $p>2$).}
 
 Similar to the previous step, we will first derive a bound for the integral in \eqref{eq:w1-integral}. The key observation is that we can compare $\orf_*$ to $(x^p)_*$ instead of $1/x$ (thanks to Lemma~\ref{lem:succsim-domination}). 

Let $q:=p/(p-1)$ denote the H\"older conjugate of $p>2$.  By Lemma~\ref{lem:succsim-domination}, 
$\orf_*(\tfrac{\alpha^2}{\scon t})\le C_\orf^q(\tfrac{\alpha^2}{t})^q\orf_*(\tfrac{\alpha^2}{\scon\alpha^2})$ for $t>\alpha^2$. Hence 
\begin{equation}\label{eq:general-p}
\int_{\alpha^2}^{1+\alpha^2} \orf_*(\tfrac{\alpha^2}{\scon t})\dd t
\le 
\frac{C_\orf^q}{q-1}\alpha^2\orf_*(\tfrac{1}{\scon}).
\end{equation}
This inequality, together with  Proposition~\ref{prop:young-decoup}, yields (Note that $C_\orf^q\le 2^2$.)
\begin{align*}  
E\sum_{i=1}^{n} \frac{|\hat{Y}_i|^3}{\lambda_{i-1}^2}
&\lesssim
\fcon\scon n+\tfrac{1}{q-1}\fcon\scon\orf_*(\tfrac{1}{\scon}).
\end{align*}
Putting $\scon=1\big/\orf_*^{-1}(n)$, this inequality becomes
\begin{align}\label{eq:y3-bound-lp-general}
E\sum_{i=1}^{n} \frac{|\hat{Y}_i|^3}{\lambda_{i-1}^2}
\lesssim
p\fcon\scon n
\by{\eqref{eq:rho-upper}}\lesssim
p\fcon\orf^{-1}(n).
\end{align}	
Hence, taking $\alpha=2\fcon\orf^{-1}(n)=2L_\orf$, inequalities \eqref{eq:y3-bound-lp-general}
 and \eqref{eq:wr-elongated} imply  
\[
\ww_1(\hat{X}_{n+1},\sn)\lesssim p\alpha.
\]
This inequality, together with Proposition~\ref{prop:trunc}\eqref{item:prop-trunc-1}, yields  Theorem~\ref{thm:ww1-new}\eqref{item:thm-ww1-new-2}.
\end{proof}

\subsection{Proof of Theorem~\ref{thm:ww2} ($\ww_2$ bounds)}\label{subsec:pf-w2}

In some sense, our proof of the $\ww_2$ bound in Theorem~\ref{thm:ww2} is an interpolation between the proofs of Theorem~\ref{thm:ww1-new} and Theorem~\ref{thm:ww3}. We observe that, to bound 
$E\sum_{i=1}^n|Y_i|^3/\lambda_i$, a better estimate than Proposition~\ref{prop:young-decoup}
can be achieved by simply bounding $|Y_i|^3/\lambda_i$ by $|Y_i|^3/C$ when $\lambda_i$ is larger than some ``threshold" $C$.
\medskip

\begin{proof}
[Proof of Theorem~\ref{thm:ww2}\eqref{item:thm-ww2-1}:] Without loss of generality, assume $s_n=1$. Set $\fcon=\onorm{Y}$. 

Define the mds $\hat{Y}=\hat{Y}^{(\infty)}$ and the  martingale $\{\hat{X}_i\}$ as in Definition~\ref{def:trunc-mart}. Note that $\hat{Y}_i=Y_i$ for all $i<T$.  

Let $\tbd>0$ be a constant to be determined later, and recall the notations $\sigma_i,\lambda_i$ in \eqref{eq:def-lambda}.
Since $\orf\succcurlyeq x^2$ and $\lambda_i\ge\tbd$ for $1\le i\le n+1$, we have
$\lambda_i^2/\orf(\lambda_i/\fcon)\le\tbd^2/\orf(\tbd/\fcon)$ for all $i$. 
 Applying Proposition~\ref{prop:v=1} and Remark~\ref{rmk:ww-normal-role}, we get
\begin{align*}
\zz_2(\hat{X}_{n+1},\sn)
&\lesssim 
\tbd^2+E\sum_{i=1}^{n}\frac{\lambda_i^2\orf(|\hat Y_i|/\fcon)}{\orf(\lambda_i/\fcon)}\\
&\lesssim
\tbd^2+E\sum_{i=1}^{n}\frac{\tbd^2\orf(|Y_i|/\fcon)}{\orf(\tbd/\fcon)}
\\&
\lesssim
\tbd^2+\frac{n\tbd^2}{\orf(\tbd/\fcon)}.
\end{align*}
Taking $\tbd=L_\orf=\fcon\orf^{-1}(n)$,  we get 
\[
\zz_2(\hat{X}_{n+1},\sn)\lesssim L_\orf^2.
\]
Finally, by  Proposition~\ref{prop:trunc}\eqref{item:prop-trunc-2} and \eqref{eq:rio}, with $\alpha=\infty$,
\[
\ww_2(X_n,\sn)
\lesssim 
\zz_2(\hat{X}_{n+1},\sn)^{1/2}+L_\orf+\norm{\sum_{i=1}^n\sigma_i^2-1}_1^{1/2}.
\]
Theorem~\ref{thm:ww2}\eqref{item:thm-ww2-1} follows.
\end{proof}

\begin{proof}[Proof of Theorem~\ref{thm:ww2}\eqref{item:thm-ww2-2}:] Without loss of generality, assume $s_n=1$.

Let $\alpha>0$ be a constant to be determined.
We define the mds $\hat{Y}=\hat{Y}^{(\alpha)}$ and the martingale $\{\hat{X}_i\}$ as in Definition~\ref{def:trunc-mart}.   Set $\fcon=\onorm{Y}$. 

\noindent{\bf Step 1.}
Applying Proposition~\ref{prop:v=1} and Remark~\ref{rmk:ww-normal-role} to the case $\tbd=\sqrt{2}\alpha$, we get
\begin{equation}\label{eq:zz2-bd-succ3}
\zz_2(\hat{X}_{n+1},\sn)
\lesssim 
\alpha^2+E\sum_{i=1}^n\frac{|\hat Y_i|^3}{\lambda_i}.
\end{equation}
Let $k>0$ be a constant to be determined later, and define and event 
\[
A_i=\{\lambda_i\le k\alpha\}.
\]
By Proposition~\ref{prop:young-decoup}, 
\begin{equation}\label{eq:sfo-hkg}
E\sum_{i=1}^n\frac{|\hat Y_i|^3}{\lambda_i}\idkt{\lambda_i\le k\alpha}
\lesssim
\fcon\scon n+\fcon\scon\alpha^{-2}\int_0^{k^2\alpha^2}\orf_*(\tfrac{\alpha^2}{\scon t^{1/2}})\dd t
\end{equation}
Since $\orf\succcurlyeq x^3$, by Lemma~\ref{lem:succsim-domination}, 
$t^{3/2}\orf_*(\tfrac{\alpha^2}{\scon t})\lesssim(k\alpha)^{3/2}\orf_*(\tfrac{\alpha^2}{\scon k\alpha})$ for $t\in(0,k\alpha)$. Thus
\begin{align*}
\int_0^{k\alpha}t\orf_*(\tfrac{\alpha^2}{\scon t})\dd t
\lesssim 
(k\alpha)^{3/2}\orf_*(\tfrac{\alpha^2}{\scon k\alpha})
\int_0^{k\alpha}t^{1-3/2}\dd t
\lesssim
(k\alpha)^2\orf_*(\tfrac{\alpha}{\scon k}).
\end{align*}
This inequality, together with \eqref{eq:sfo-hkg}, yields
\[
E\sum_{i=1}^{n}
\frac{|\hat Y_i|^3}{\lambda_i}\idkt{A_i}
\lesssim
\fcon\scon n+\fcon\scon k^2\orf_*(\tfrac{\alpha}{\scon k}).
\]
Taking $\scon$ such that $n=4k^2\orf_*(\tfrac{\alpha}{\scon k})$, i.e.,
\[
\rho=\frac{\alpha/k}{\orf_*^{-1}(n/k^2)}\by{\eqref{eq:orlicz-ineq}}\le \frac{4\alpha k}{n}\orf^{-1}(\frac{n}{4k^2}),
\]
we obtain
\begin{equation}\label{eq:y3lambda-ai}
E\sum_{i=1}^{T-1}
\frac{|\hat Y_i|^3}{\lambda_i}\idkt{A_i}
\lesssim 
\fcon\scon n
\lesssim \fcon\alpha k\orf^{-1}(\frac{n}{4k^2}).
\end{equation}

\noindent{\bf Step 2.}
 By Lemma~\ref{lem:orfnorm-z},
$\norm{\{\hat{Y}_i\}_{i=1}^n}_3^3\lesssim 
\norm{Y}_3^3$. 
Recalling that $A_i^c=\{\lambda_i>k\alpha\}$ and  $\orf\succcurlyeq x^3$, 
\begin{align*}
E\sum_{i=1}^{n}
\frac{|\hat Y_i|^3}{\lambda_i}\idkt{A_i^c}
\le 
E\sum_{i=1}^{n}
\frac{|\hat Y_i|^3}{k\alpha}
\lesssim 
\frac{n\norm{Y}_3^3}{k\alpha}
\lesssim_{\orf(1)} 
\frac{n\fcon^3}{k\alpha}.
\end{align*}
where Lemma~\ref{alem:norm-dominate-ineq} is used in the last inequality. 
This inequality, together with \eqref{eq:y3lambda-ai}, yields
\[
E\sum_{i=1}^{n}
\frac{|\hat Y_i|^3}{\lambda_i}
\lesssim_{\orf(1)} 
\fcon\alpha k\orf^{-1}(\frac{n}{4k^2})
+\frac{n\fcon^3}{k\alpha}.
\]
Let $g(x):=\frac{1}{x}\orf(x)$. 
Choosing $k$ such that $\fcon\alpha k\orf^{-1}(\frac{n}{4k^2})=\frac{n\fcon^3}{k\alpha}$, i.e.,
\[
\frac{1}{\alpha k}=\sqrt{g^{-1}(\tfrac{\alpha^2}{4\fcon^2})/(n\fcon^2)},
\]
we arrive at the bound
\begin{equation*}
E\sum_{i=1}^{n}
\frac{|\hat Y_i|^3}{\lambda_i}
\lesssim_{\orf(1)} 
\fcon^2\sqrt{ng^{-1}(\tfrac{\alpha^2}{4\fcon^2})}.
\end{equation*}

\noindent{\bf Step 3.}
Further, by \eqref{eq:zz2-bd-succ3}, we get
\[
\zz_2(\hat{X}_{n+1},\sn)
\lesssim_{\orf(1)} 
\alpha^2+\fcon^2\sqrt{ng^{-1}(\tfrac{\alpha^2}{\fcon^2})}\,.
\]
Thus, choosing $\alpha=2L_\orf=2\fcon\orf^{-1}(n)$,  by Proposition~\ref{prop:trunc}\eqref{item:prop-trunc-1} and \eqref{eq:rio}, 
\begin{align*}
\ww_2(X_n,\sn)
&\lesssim 
\zz_2(\hat{X}_{n+1},\sn)^{1/2}+L_\orf+\norm{V_n^2-1}_1^{1/2}\\
&\lesssim_{\orf(1)} 
L_\orf+
\fcon\left[ng^{-1}(\orf^{-1}(n)^2)\right]^{1/4}+\norm{V_n^2-1}_1^{1/2}.
\end{align*}
Theorem~\ref{thm:ww2}\eqref{item:thm-ww2-2} is proved.
\end{proof}

\section{Optimality of the $\ww_1$ rates: Proof of Proposition~\ref{prop:optimal}}
\label{sec:optimal}

For any $p>2$,  we will construct a mds $\{Y_i\}_{i=1}^n$ such that $V_n^2=1$ and
$\ww_1(X_n,\sn)\gtrsim_pL_p(Y)$. Note that in our examples,  $L_p(Y)\to 0$ as $n\to\infty$, justifying the optimality of  the term $L_\orf$ in Theorem~\ref{thm:ww1-new}.

We choose 
\[
\alpha=\frac{1}{\log n}.
\]
(Actually for any $p>2$, any $\alpha\gg n^{1/p-1/2}$ so that $\alpha_n\to 0$ as $n\to\infty$ would work.)

\begin{example}\label{ex1}
For $n\ge 3$, let $Y_1,\ldots,Y_{n-2}, Y_n,\xi,\eta$ be independent random variables such that 
\begin{itemize}
\item $Y_n,\xi\sim\sn(0,\alpha_n^2)$, 
\item $Y_1,\ldots,Y_{n-2}$ are i.i.d. $\sn(0,\tfrac{1-2\alpha_n^2}{n-2})$ normal random variables, 
\item  $\eta$ has distribution
\[
P(\eta=\tfrac{1}{2})=\tfrac{4}{5},
\quad
P(\eta=-2)=\tfrac{1}{5},
\]
\end{itemize}
We let $X_m:=Y_1+\ldots+Y_m$ and define a set $A\subset\R$ as
\[
 A=\left\{x\in\R:\cos x\ge\tfrac{1}{2}\right\}=\bigcup_{k\in\Z}\big[(2k-\tfrac13)\pi,(2k+\tfrac13)\pi\big].
\]
Define
\[
Y_{n-1}=\xi\idkt{X_{n-2}\notin\alpha A}
+
\alpha\eta\idkt{X_{n-2}\in\alpha A}.
\] 
\end{example}

Clearly, $\{X_m\}_{m=1}^n$ is a martingale, and 
\[
\sigma_{n-1}^2=\sigma_{n}^2=\alpha^2, \quad
\sigma_i^2=\tfrac{1-2\alpha^2}{n-2} \,\text{ for }i=1,\ldots,n-2.
\]
Of course, $V_n^2=\sum_{i=1}^n\sigma_i^2=1$ almost surely, and 
\[
L_p=\left(
\sum_{i=1}^n \ee{|Y_i|^p}
\right)^{1/p}
\le 
\left(
(n-2)\frac{c}{n^{p/2}}+c\alpha^p
\right)^{1/p}
\lesssim_p\alpha.
\]


\begin{proof}[Proof of Proposition~\ref{prop:optimal}\eqref{item:prop-opt-1}]
Let $\{Y_i\}_{i=1}^n$ be the mds in Example~\ref{ex1}. 
We consider the function 
\[
h(x)=\alpha\sin(\tfrac{x}{\alpha}).
\]
Clearly, $[h]_{0,1}\le 1$. Recall the definition of the function $h_\alpha(x)$ in \eqref{eq:stein}. Then
\begin{equation}\label{eq:calcul-h}
h_\alpha(x)=\alpha\int_{\R}\sin\big(\tfrac{1}{\alpha}(x-\alpha u)\big)\phi(u)\dd u
=h(x)\int_\R\cos u\phi(u)\dd u=\frac{1}{\sqrt{e}}h(x),
\end{equation}
where $\phi(x)=\tfrac{1}{\sqrt{2\pi}}e^{-x^2/2}$ denotes the standard normal density.
Moreover,
\begin{align*}
\ee{h(X_n)-h(\sn)}
=\ee{
h_\alpha(X_{n-1})-h_\alpha(X_{n-2}+\alpha\sn)
}.
\end{align*}
By the definition of $Y_{n-1}$, we have
\begin{align*}
\ee{
h_\alpha(X_{n-1})}
=\ee{h_\alpha(X_{n-2}+\alpha\sn)\idkt{X_{n-2}\notin\alpha A}}+\ee{h_\alpha(X_{n-1}+\alpha\eta)\idkt{X_{n-2}\in\alpha A}}.
\end{align*}
Thus
\begin{align}\label{eq:comp-telesc}
&\ee{h(X_n)-h(\sn)}\\
\MoveEqLeft=
\bee{
\left(
h_\alpha(X_{n-2}+\alpha\eta)-h_\alpha(X_{n-2}+\alpha\sn)
\right)\idkt{X_{n-1}\in\alpha A}}\nn\\
\MoveEqLeft
\by{\eqref{eq:calcul-h}}=\tfrac{\alpha}{\sqrt{e}}
\bee{
\left(\sin\tfrac{X_{n-2}}{\alpha}(\cos\eta-\cos\sn)+\cos\tfrac{X_{n-2}}{\alpha}(\sin\eta-\sin\sn)
\right)\idkt{X_{n-2}\in\alpha A}
}\nn\\
\MoveEqLeft
=\tfrac{\alpha}{\sqrt{e}}\left(
\bee{
\sin\tfrac{X_{n-2}}{\alpha}
\idkt{X_{n-2}\in\alpha A}
}\ee{\cos\eta-\cos\sn}
+\bee{
\cos\tfrac{X_{n-2}}{\alpha}
\idkt{X_{n-2}\in\alpha A}
}\ee{\sin\eta}
\right).\nn
\end{align}
Since the set $A$ is symmetric, i.e., $A=-A$, and $X_{n-2}\sim\sn(0,1-2\alpha^2)$, we get $\bee{
\sin\tfrac{X_{n-2}}{\alpha}
\idkt{X_{n-2}\in\alpha A}
}=0$.  Hence
\begin{align*}
\ee{h(X_n)-h(\sn)}
&=
\tfrac{\alpha}{\sqrt{e}}
\bee{
\cos\tfrac{X_{n-2}}{\alpha}
\idkt{X_{n-2}\in\alpha A}
}\ee{\sin\eta}\\
&\ge 
\tfrac{\alpha}{2\sqrt{e}}P(X_{n-2}\in\alpha A)\ee{\sin\eta}.
\end{align*}
Note that $E[\sin\eta]=\tfrac{4}{5}\sin\tfrac{1}{2}-\tfrac{1}{5}\sin 2>0.2$, and, writing
$c_1:=(1-2\alpha^2)^{-1/2}$,
\begin{align}\label{eq:xina}
P(X_{n-2}\in\alpha A)
&=P(\sn\in c_1\alpha A)\nn\\
&=
\sum_{k\in\Z}
\int_{c_1(6k-1)\alpha\pi/3}^{c_1(6k+1)\alpha\pi/3}\phi(u)\dd u\nn\\
&\ge 
\frac{1}{2}\left(
\int_{0}^{c_1\alpha\pi/3}\phi(u)\dd u+\sum_{k=0}^\infty
\tfrac{2c_1\alpha\pi}{3}\phi(\tfrac{c_1(6k+1)\alpha\pi}{3})
\right)\nn\\
&\ge 
\frac{1}{6}\int_0^\infty\phi(u)\dd u=\frac{1}{12}.
\end{align}
Therefore, $\ee{h(X_n)-h(\sn)}\ge \frac{\alpha}{120\sqrt{e}}$ and so
\[
\ww_1(X_n,\sn)\ge \frac{\alpha}{120\sqrt{e}}\ge CL_p,
\]
justifying the optimality of the power of $L_\orf$ in Theorem~\ref{thm:ww1-new}. 
\end{proof}

We can modify Example~\ref{ex1} above to justify the optimality of the exponent of $\norm{V^2-1}_{1/2}$ in Theorem~\ref{thm:ww1-new}.

\begin{example}\label{ex2}
For $n\ge 3$, let $Y_1,\ldots, Y_{n-2}$ be i.i.d. with $Y_1\sim\sn(0,\frac{1-\alpha^2}{n-2})$. Let 
\[
Y_{n-1}=\alpha\eta\idkt{X_{n-2}\in \alpha A},
\] 
where $X_{n-2}=\sum_{i=1}^{n-1}Y_i$, $\alpha=\alpha_n=1/\log n$ , $\eta$ and the set $A\subset\R$ are as in Example \ref{ex1}. Let $Y_n\sim\sn\left(0,\alpha^2P(\sn\notin\kappa A)\right)$, where $\kappa=\kappa_n=\alpha/\sqrt{1-\alpha^2}$.
\end{example}

\begin{proof}[Proof of Proposition~\ref{prop:optimal}\eqref{item:prop-opt-2}]
Let the mds $\{Y_i\}_{i=1}^n$ be as in Example~\ref{ex2}.

Clearly, $L_3=\left(
\sum_{i=1}^n\ee{|Y_i|^3}
\right)^{1/3}\lesssim \alpha$, and 
\begin{align*}
\norm{V_n^2-1}_{1/2}
&=\alpha^2\bee{|\idkt{X_{n-2}\in\alpha A}-P(X_{n-2}\in\alpha A)|^{1/2}}^2\\
&\asymp \alpha^2P(\sn\notin\kappa A)P(\sn\in\kappa A).
\end{align*}
Hence, by Theorem~\ref{thm:ww1-new}\eqref{item:thm-ww1-new-2}, we get
\begin{equation}\label{eq:counter-2}
\ww_1(X_n,\sn)\lesssim \alpha,
\end{equation}
Next we will show that, there exists $N>0$ such that for $n>N$,
\begin{equation}\label{eq:v2-1}
\norm{V_n^2-1}_{1/2}\asymp \alpha^2.
\end{equation}
Note that the same computation as in \eqref{eq:xina} yields $P(\sn\in\kappa A)\ge 1/12$.
Hence we only need to show that $P(\sn\notin\kappa A)\gtrsim 1$. Indeed,
\begin{align*}
P(\sn\in\kappa A)
&=\sum_{i\in\Z}
\int_{c_1(6i-1)\kappa\pi/3}^{c_1(6i+1)\kappa\pi/3}\phi(u)\dd u\nn\\
&\le 
2\int_0^{c_1\kappa\pi/3}\phi(u)\dd u+2\sum_{i=1}^\infty\frac{2c_1\kappa\pi}{3}\phi(\frac{c_1(6i-1)\kappa\pi}{3})\\
&\le 
2\int_0^{c_1\kappa\pi/3}\phi(u)\dd u+\frac{2}{3}\int_{-c_1\kappa\pi/3}^\infty\phi(u)\dd u\\
&\le \frac{8}{3}\int_0^{c_1\kappa\pi/3}\phi(u)\dd u+\frac{1}{3}\le \frac{1}{2}
\end{align*}
for all $n$ sufficiently big. Inequality \eqref{eq:v2-1} follows. 

It remains to show that
\[
\ww_1(X_n,\sn)\gtrsim\alpha.
\]
To this end, we consider the function $h(x)=\alpha\sin(\tfrac{x}{\alpha})$.  
Define $\lambda>0$ by
\[
\lambda^2=\var(Y_n)=\alpha^2P(\sn\notin\kappa A)\in[\tfrac{1}{2}\alpha^2,\tfrac{11}{12}\alpha^2].
\]
Recall the definition of  $h_\lambda(x)$ in \eqref{eq:stein}.  By the same calculation as in  \eqref{eq:calcul-h}, 
\[
h_{\lambda}(x)=h(x)\int_\R\cos(\tfrac{\lambda}{\alpha}u)\phi(u)\dd u
=h(x)\exp(-\frac{\lambda^2}{2\alpha^2}).
\]
By  symmetry, $E[h(\sn)]=0$ and $\ee{h_\lambda(X_{n-2})\idkt{X_{n-2}\notin\alpha A}}=0$. Hence
\begin{align*}
&\ee{h(X_n)-h(\sn)}\\
&=\ee{h_\lambda(X_{n-2})\idkt{X_{n-2}\notin\alpha A}}
+\ee{h_\lambda(X_{n-2}+\alpha\eta)\idkt{X_{n-2}\in\alpha A}}\\
&=\alpha\exp(-\frac{\lambda^2}{2\alpha^2})\ee{
\cos(\tfrac{X_{n-2}}{\alpha})\sin\eta\idkt{X_{n-2}\in\alpha A}
}\\
&\ge 0.1\alpha\exp(-\frac{\lambda^2}{2\alpha^2})P(X_{n-2}\in\alpha A)\ge c\alpha.
\end{align*}
Display \eqref{eq:counter-2} is proved. Our proof of the Proposition is complete.
\end{proof}

\section{Some open questions}\label{sec:qs}

\begin{enumerate}
\item  
For the $\ms L^\infty$ case (i.e. the mds is uniformly $\ms L^\infty$-bounded), is the typical $\ww_1$ rates $O(n^{-1/2}\log n)$ within Theorem~\ref{thm:ww1-new} optimal? Recall that
it is shown by Bolthausen \cite{Bol-82} that this  rate is  typically optimal for the Kolmogorov distance.

\item 
For the $\ms L^\infty$ case, are the typical $\ww_r$ rates $O(n^{-/(2r)})$, $r\in\{2,3\}$, in Theorems~\ref{thm:ww2} and \ref{thm:ww3} optimal?
 For general $r>3$ and bounded mds, can we get similar bounds  $O(n^{-/(2r)})$, $r>3$?

\item Can we say anything about the optimality of the $\ww_2$ rates in Theorem~\ref{thm:ww2} when
 the mds is $\ms L^p$ integrable, for any $p>2$?
 
\item 
Is there a better (unified) formula than the ``piecewise" $\ww_2$ bound in Theorem~\ref{thm:ww2}  for $\ms L^p$ martingales, $p>2$?

\item As a feature of the $\ww_1,\ww_2$ bounds in Theorems~\ref{thm:ww1-new} and \ref{thm:ww2}, better integrability implies faster (typical) convergence rates for the martingale CLT. But this is no longer the case for the $\ww_3$ estimate in Theorem~\ref{thm:ww3}.  Is it possible to get better $\ww_3$ bounds for general martingales with better integrability than $\ms L^3$?

\item How to obtain the Wasserstein rates for the CLT of multi-dimensional martingales? Can we say something on  the dependence of the rates on the dimension?

\item How to obtain the Wasserstein-$r$ convergence rates, $r>3$, for the martingale CLT in terms of $L_\orf$?
\end{enumerate}

\newpage
\newtheorem{atheorem}{Theorem}
\numberwithin{atheorem}{section}
\newtheorem{alemma}[atheorem]{Lemma}
\newtheorem{acorollary}[atheorem]{Corollary}

\appendix
\section{Appendix}
\subsection{Comparison between Orlicz norms}
Recall the definition of the (mean-)orlicz norm $\onorm{\cdot}$ for a sequence in Definition~\ref{def:orlicz-norm}.
\begin{alemma}
\label{alem:norm-dominate-ineq}
Let $p\ge 1$ and let $Y=\{Y_i\}_{i=1}^n$ be a sequence of random variables with $\norm{Y}_p<\infty$. 
If an N-function $\orf$ satisfies $\orf\succcurlyeq x^p$, then 
\[
\norm{Y}_p\le \left(1+\frac{1}{\orf(1)}\right)^{1/p}\onorm{Y}.
\]
\begin{proof}
Without loss of generality assume $\onorm{Y}=1$.
 Since $\orf/x^p$ is increasing,
\begin{align*}
\ee{|Y_i|^p}
\le 
1+\ee{|Y_i|^p\idkt{|Y_i|>1}}
\le 
1+\frac{1}{\orf(1)}\bee{\orf(|Y_i|)}.
\end{align*}
Thus, by the definition of $\onorm{\cdot}$,
\begin{align*}
\norm{Y}_p^p
=
\frac{1}{n}\sum_{i=1}^n\ee{|Y_i|^p}
\le 
1+\frac{1}{\orf(1)n}\sum_{i=1}^n\ee{\orf(|Y_i|)}
\le 1+\frac{1}{\orf(1)}.
\end{align*}
The lemma follows.
\end{proof}
\end{alemma}
\subsection{Regularity of Gaussian smoothing:Proof of Lemma~\ref{lem:smoothing}}
The proof is exactly as in \cite[Lemma 6.1]{DMR-09}.
\begin{proof}
Since $f_\sigma(x)=\int_\R f(x-\sigma u)\phi(u)$, integration by parts yields
\begin{align*}
f_\sigma^{(k)}(x)&=\frac{1}{\sigma^k}\int_\R f^{(\ell)}(x-\sigma u)\phi^{(k-\ell)}(u)\dd u\\
&=\frac{1}{\sigma^{k-\ell}}\int_\R [f^{(\ell)}(x-\sigma u)-f^{(\ell)}(x)]\phi^{(k-\ell)}(u)\dd u, \quad\forall 0\le\ell<k.
\end{align*}	
The lemma follows by taking $\ell=r-1$ and using the fact $[f]_{r-1,1}=1$.
\end{proof}

\subsection{Moment bound of the maximum}
\begin{alemma}
\label{alem:square-by-orlicz}
Let $\orf\succcurlyeq x^r$ be an N-function, $r>0$. 
For any sequence of random variables $Y=\{Y_i\}_{i=1}^n$,  we have
\[
\norm{\max_{i=1}^n |Y_i|}_r\le 2^{1/r}\onorm{Y}\orf^{-1}(n).
\]
\end{alemma}
\begin{proof}
Without loss of generality, assume $\onorm{Y}=1$. 
For any  $\alpha>0$,
\[
|Y_i|^r\idkt{|Y_i|\ge\alpha}
\le 
\frac{\alpha^r}{\orf(\alpha)}\orf(|Y_i|)\idkt{|Y_i|\ge\alpha}, \quad 1\le i\le n
\]
where we used the fact that $x\mapsto\frac{\orf}{x^r}$ is increasing.
Hence
\begin{align*}
 \ee{\max_{i=1}^n |Y_i|^r}
 &\le \alpha^r+\ee{\max_{i=1}^n |Y_i|^r\idkt{|Y_i|\ge\alpha}}\\
 &\le \alpha^r+\bee{\sum_{i=1}^n\frac{\alpha^r}{\orf(\alpha)}\orf(|Y_i|)}
 \le \alpha^r+\frac{n\alpha^r}{\orf(\alpha)}
 \end{align*}	
 
 Taking $\alpha=\orf^{-1}(n)$, we get $\norm{\max_{i=1}^n |Y_i|}_r^r\le 2\alpha^r$.
\end{proof}

\subsection{Proof of Lemma~\ref{lem:succsim-domination}}

\begin{proof}
We only give the proof of \eqref{item:lem-succ}.  The proof of \eqref{item:lem-prec} is similar.

The case $\orf=x^p$ is trivial, so we only consider the general case $\orf\preccurlyeq x^p$. It suffices to show that, for any $s>t>0$,
\[
\frac{s^{1/q}}{\orf_*^{-1}(s)}\ge \frac{1}{4}\frac{t^{1/q}}{\orf_*^{-1}(t)},
\]
which, writing $F(x):=\orf^{-1}(x)\orf_*^{- 1}(x)$, is equivalent to 
\[
\frac{\orf^{-1}(s)/s^{1/p}}{\orf^{-1}(t)/t^{1/p}}\ge \frac{1}{C_\orf}\frac{F(s)/s}{F(t)/t} \quad\forall s>t>0.
\]
Note that $\orf\preccurlyeq x^p$ implies that the left-side is at least 1. Lemma~\ref{lem:succsim-domination}\eqref{item:lem-succ} follows.
\end{proof}

\subsection{A Burkholder inequality}
\begin{atheorem}
\label{athm:hall-heyde}
Let $\{Y_i\}_{i=1}^n$ be a mds and let $S_m=Y_1+\ldots+Y_m$, $1\le m\le n$. Recall the notation $\sigma_i(Y)^2$ in \eqref{def:notations}. Then, fir $p>0$,
\[
\ee{\max_{i=1}^n|S_n|^p}
\lesssim_p
\ee{\big(\sum_{i=1}^n\sigma_i^2(Y)\big)^{p/2}}+\ee{\max_{i=1}^n|Y_i|^p}.
\]
\end{atheorem}
This version is taken from \cite[Theorem~2.11]{Hall-Heyde-80}. For a more general inequality, see \cite[Theorem 21.1]{BDG-72}.

\end{document}